\documentclass[12pt]{article}

\usepackage{amsmath,amssymb,amsfonts,amsthm,graphicx}
\usepackage{bm}
\usepackage{dsfont}
\usepackage[margin=1in]{geometry}
\usepackage{float}
\usepackage{tocloft, color, soul}
\usepackage{upgreek}
\usepackage[OT1]{fontenc}
\usepackage{tikz}
\usepackage{pdflscape}
\usetikzlibrary{arrows.meta,calc}
\usepackage[colorlinks=true,linkcolor=blue,citecolor=blue,urlcolor=blue]{hyperref}

\graphicspath{{./}}

\newcommand{\R}{\mathbb R}
\newcommand{\E}{\mathds{E}}
\renewcommand{\P}{\mathds P}

\newcommand{\eps}{\varepsilon}
\newcommand{\Leb}{\operatorname{Leb}}
\newcommand{\pat}{\operatorname{pat}}
\newcommand{\rv}[1]{\bm{#1}}
\newcommand{\X}{\rv{X}}
\newcommand{\Y}{\rv{Y}}

\newtheorem{theorem}{Theorem}
\newtheorem{lemma}[theorem]{Lemma}
\newtheorem{proposition}[theorem]{Proposition}
\newtheorem{corollary}[theorem]{Corollary}

\newtheorem{example}[theorem]{Example}

\title{Permutations from Random Walk}
\author{
Peter E. Francis\thanks{\raggedright Department of Mathematics, Stony Brook University,
Stony Brook, NY 11794.
\mbox{E-mail: \href{mailto:peter.e.francis@stonybrook.edu}{\nolinkurl{peter.e.francis@stonybrook.edu}}}.
Funded by DMS-2346986.}
\and Peter Winkler\thanks{\raggedright Department of Mathematics, Dartmouth
College, Hanover, NH 03755.
\mbox{E-mail: \href{mailto:peter.winkler@dartmouth.edu}{\nolinkurl{peter.winkler@dartmouth.edu}}}}}
\date{}

\begin{document}

\maketitle

\begin{abstract}
Xavier and Yushi run a ``random race'' as follows.  An atomless probability
distribution $\mu$ on the real line is chosen.  The runners begin at zero.  At
time $i$ Xavier draws $\X_i$ from $\mu$ and advances that distance, while Yushi
advances by an independent drawing $\Y_i$.  After $n$ such moves, what is the
probability that Yushi led all the way?

That the answer (namely, $4^{-n}\binom{2n}{n}$) is independent of $\mu$ follows from
a classical theorem of Darling, stating that for symmetric atomless increments, the
distribution of each individual rank in the permutation obtained by ranking
the partial sums is independent of the step law.  We give a self-contained proof
and extend the result to the permutations generated by partial sums of uniformly
random signed permutations of any fixed, finite, generic set of reals.

For atomless increments with mean zero and finite variance, without assuming
symmetry, we show that random-walk permutations approach a random object that we call the
``Wiener permuton,'' whose expected pattern densities equal the probabilities of the corresponding permutations generated by finite random walks with
centered Laplace increments.

Finally, we exhibit an infinite family of constructions whose limiting permutons
interpolate between the Wiener permuton and the recursive separable permuton; each
has the same intensity permuton, providing a single two-dimensional extension of the
classical arcsine law for all of them.
\end{abstract}

\newpage
\tableofcontents

\clearpage
\newcommand{\proofdependencydiagram}{%
\begin{landscape}
\thispagestyle{plain}
\begin{center}
{\Large\bfseries Proof-dependency diagram}\par
\vspace{3pt}
\resizebox{0.99\linewidth}{!}{%
\begin{tikzpicture}[
  x=1cm,
  y=1cm,
  result/.style={
    draw=black!70,
    fill=white,
    rounded corners=1pt,
    align=center,
    text width=2.5cm,
    minimum height=0.72cm,
    inner sep=2pt,
    font=\fontsize{7}{8}\selectfont
  },
  external/.style={result,densely dashed},
  independent/.style={result,fill=black!2},
  theorempart/.style={result,fill=black!6,text width=3.2cm},
  dependency/.style={-{Latex[length=1.8mm,width=1.2mm]},
    draw=black!75,line width=0.4pt,rounded corners=1.2pt},
  implicit/.style={dependency,densely dashed},
  region/.style={draw=black!35,rounded corners=2pt,line width=0.5pt}
]

\draw[region] (-1.75,7.00) rectangle (16.85,14.30);
\node[anchor=north west,font=\small\bfseries]
  at (-1.55,14.18) {Finite ranks and counting};
\draw[region] (-1.75,-0.50) rectangle (16.85,6.70);
\node[anchor=north west,font=\small\bfseries]
  at (-1.55,6.58) {Permuton limits};
\draw[region] (17.35,-0.50) rectangle (22.05,14.30);
\node[anchor=north west,font=\small\bfseries]
  at (17.55,14.18) {Theorem~\ref*{thm:main-results}};

\node[result] (L11) at (-0.10,13.00)
  {L\ref*{lem:rank-convolution}\\Rank convolution};
\node[result] (P20) at (3.00,13.00)
  {P\ref*{prop:generating-function}\\Generating function};
\node[result] (L12) at (-0.10,11.10)
  {L\ref*{lem:generic-ballot}\\Generic ballot};
\node[result] (C13) at (3.00,11.10)
  {C\ref*{cor:first-rank-invariant}\\First-rank invariance};
\node[result] (T14) at (6.20,12.00)
  {T\ref*{thm:generic-set-ranks}\\Generic signed permutations};
\node[result] (T4) at (9.40,12.60)
  {T\ref*{thm:darling}\\Darling formula};
\node[result] (L15) at (12.50,12.60)
  {L\ref*{lem:intensity-matrix}\\Intensity matrix};
\node[result] (T16) at (15.50,12.60)
  {T\ref*{thm:intensity-limit}\\Random-walk intensity};
\node[result] (L17) at (-0.10,9.30)
  {L\ref*{lem:B-sets-reformatted}\\Endpoint path count};
\node[result] (P18) at (3.00,9.30)
  {P\ref*{prop:A-B-bijection-reformatted}\\Negative edges / last zero};
\node[result] (C19) at (6.20,9.30)
  {C\ref*{cor:negative-edge-count}\\Negative-edge count};
\node[result] (T3) at (9.40,9.30)
  {T\ref*{thm:arcsine}\\Arcsine law};
\node[result] (T1) at (12.50,9.30)
  {T\ref*{thm:race}\\Leading all the way};
\node[result] (L2) at (15.50,9.30)
  {L\ref*{lem:central-binomial-path-count}\\Central-binomial path count};

\draw[dependency] (L11) -- (P20);
\draw[dependency] (L11.north) -- ++(0,0.24) -| ([xshift=-0.35cm]T4.north);
\draw[dependency] (L11.south) -- (-0.10,12.20)
  -- ([yshift=0.20cm]T14.west);
\draw[dependency] (L12) -- (C13);
\draw[dependency] (L12.east) -- (1.45,11.10) -- (1.45,12.00)
  -- (T14.west);
\draw[dependency] (C13.east) -- (4.55,11.10) -- (4.55,11.80)
  -- ([yshift=-0.20cm]T14.west);
\draw[dependency] (C13.south) -- (3.00,10.35) -- (7.85,10.35)
  -- (7.85,12.60) -- (T4.west);
\draw[dependency] (T4) -- (L15);
\draw[dependency] (L15) -- (T16);
\draw[dependency] (T4.north) -- ++(0,0.49) -| (T16.north);
\draw[dependency] (L17) -- (P18);
\draw[dependency] (P18) -- (C19);
\draw[dependency] (C19) -- (T3);
\draw[dependency] (C13.south) -- (3.00,9.78) -- (7.80,9.78)
  -- (7.80,9.50) -- ([yshift=0.20cm]T3.west);
\draw[dependency] (L12.west) -- (-1.55,11.10) -- (-1.55,7.25)
  -- (9.40,7.25) -- (T3.south);
\draw[dependency] (T3) -- (T4);
\draw[dependency] (T3) -- (T1);
\draw[dependency] ([yshift=-0.18cm]T4.east) -- (10.95,12.42)
  -- (10.95,9.50) -- ([yshift=0.20cm]T1.west);
\draw[implicit] (L2) -- (T1);

\node[external] (T22) at (-0.10,5.40)
  {T\ref*{thm:dyadic-wiener-convergence}\\Dyadic-to-Wiener convergence};
\node[external] (L23) at (-0.10,3.70)
  {L\ref*{lem:wiener-occupation-atomless}\\Atomless Wiener occupation};
\node[result] (P24) at (3.00,4.80)
  {P\ref*{prop:dyadic-rank-transform-uniform}\\Uniform rank transform};
\node[result] (P25) at (6.20,5.60)
  {P\ref*{prop:dyadic-graph-convergence}\\Graph-measure convergence};
\node[result] (T26) at (6.20,3.80)
  {T\ref*{thm:dyadic-permutation-process}\\Dyadic permutation process};
\node[result] (T28) at (9.40,5.60)
  {T\ref*{thm:three-density-convergence}\\Three-density convergence};
\node[result] (T8) at (9.40,3.80)
  {T\ref*{thm:random-walk-permuton-limit}\\Random-walk permuton limit};
\node[external] (T5) at (12.50,5.60)
  {T\ref*{thm:pattern-continuity}\\Permutation limits};
\node[result] (Llex) at (15.50,5.60)
  {L\ref*{lem:binary-lexicographic-order}\\Lexicographic order};
\node[result] (C27) at (12.50,3.80)
  {C\ref*{cor:wiener-intensity}\\Wiener intensity};
\node[external] (L6) at (-0.10,1.40)
  {L\ref*{lem:random-pattern-equivalence}\\Random pattern densities};
\node[result] (C7) at (3.00,1.40)
  {C\ref*{cor:random-permuton-determined}\\Random permuton determined};
\node[independent] (T21) at (6.20,0.45)
  {T\ref*{thm:inversion}\\Inversion failure};
\node[independent] (P30) at (9.40,0.45)
  {P\ref*{lem:powers-inflation}\\Powers inflation};
\node[result] (T29) at (9.40,2.25)
  {T\ref*{thm:laplace-patterns}\\Wiener patterns / Laplace walk};
\node[external] (T31) at (12.50,1.70)
  {T\ref*{thm:separable-limit}\\Separable limit and intensity};

\draw[dependency] (T22.east) -- (1.45,5.40) -- (1.45,4.98)
  -- ([yshift=0.18cm]P24.west);
\draw[dependency] (L23.east) -- (1.45,3.70) -- (1.45,4.62)
  -- ([yshift=-0.18cm]P24.west);
\draw[dependency] (P24.north) |- (P25.west);
\draw[dependency] (P24.south) |- (T26.west);
\draw[dependency] (L23.east) -- (1.55,3.70) -- (1.55,3.62)
  -- ([yshift=-0.18cm]T26.west);
\draw[dependency] (P24.south) -- (3.00,3.05) -- (9.40,3.05)
  -- (T8.south);
\draw[dependency] (T26) -- (T8);
\draw[implicit] (L23.south) -- (-0.10,2.80) -- (8.90,2.80)
  -- ([xshift=-0.50cm]T8.south);
\draw[dependency] (T8) -- (C27);
\draw[dependency] (P25) -- (T28);
\draw[dependency] (T26.east) -- (7.80,3.80) -- (7.80,5.42)
  -- ([yshift=-0.18cm]T28.west);
\draw[dependency] (T5) -- (T28);
\draw[implicit] (L6) -- (C7);

\node[theorempart] (Mmainb) at (19.70,11.60)
  {\ref*{thm:main-results}(b)\\Common intensity};
\node[theorempart] (Mmaina) at (19.70,8.60)
  {\ref*{thm:main-results}(a)\\Finite-$L$ convergence};
\node[theorempart] (Mmainc) at (19.70,4.40)
  {\ref*{thm:main-results}(c)\\Pairwise distinction};
\node[theorempart] (Mmaind) at (19.70,1.70)
  {\ref*{thm:main-results}(d)\\Separable endpoint};

\draw[dependency] (T14.north) -- (6.20,13.95) -- (17.20,13.95)
  -- (17.20,11.78)
  -- ([yshift=0.18cm]Mmainb.west);
\draw[dependency] (L15.north) -- (12.50,13.32) -- (17.05,13.32)
  -- (17.05,11.60)
  -- (Mmainb.west);
\draw[dependency] (T16.east) -- (16.90,12.60) -- (16.90,11.42)
  -- ([yshift=-0.18cm]Mmainb.west);
\draw[dependency] (Mmaina) -- (Mmainb);
\draw[dependency] (T8.east) -- (10.95,3.80) -- (10.95,6.41)
  -- (16.95,6.41)
  -- (16.95,8.60) -- (Mmaina.west);
\draw[dependency] (Llex.east) -- (17.25,5.60) -- (17.25,8.42)
  -- ([yshift=-0.18cm]Mmaina.west);
\draw[dependency] (T5.south) -- (12.50,4.58) -- (17.15,4.58)
  -- ([yshift=0.18cm]Mmainc.west);
\draw[dependency] (T29.east) -- (10.95,2.25) -- (10.95,2.60)
  -- (17.20,2.60) -- (17.20,4.22)
  -- ([yshift=-0.18cm]Mmainc.west);
\draw[dependency] (T31.north) -- (12.50,2.30) -- (21.50,2.30)
  -- (21.50,4.40) -- (Mmainc.east);
\draw[dependency] (T31) -- (Mmaind);
\draw[dependency] (T31.east) -- (14.15,1.70) -- (14.15,-0.25)
  -- (21.75,-0.25) -- (21.75,11.60) -- (Mmainb.east);

\end{tikzpicture}%
}

\vspace{3pt}
\begin{minipage}{0.97\linewidth}
\footnotesize
An arrow $A\to B$ means that $A$ is used in the proof of $B$.
Dashed boxes mark results cited from the literature; dashed arrows mark an
implicit, immediate, or special-case dependency.  Disconnected boxes mark
results that do not feed a later numbered proof.  The two examples and
unnumbered external tools are omitted.
\end{minipage}
\end{center}
\end{landscape}
}

\section{Introduction}

\subsection{A Random Race}

Imagine two runners, Xavier and Yushi.  On each of $n$ turns Xavier draws a random number $\rv{X}_i$
from an atomless distribution $\mu$ and adds $\rv{X}_i$ to his current position. At the same time, Yushi draws
an independent number $\rv{Y}_i$ from the same distribution and adds that to her present position.
Of course, the probability that Yushi wins the race is 1/2. What is the probability that she led the whole way?
That is, what is the value of
\[
\rv{Y}_1+\cdots+\rv{Y}_k>\rv{X}_1+\cdots+\rv{X}_k
\qquad \hbox{for } 1\le k\le n~?
\]
At first sight the choice of $\mu$ ought to matter.  Long-tailed steps should
feel different from short, steady steps; positive steps should feel different
from steps that sometimes move a runner backward.  But for this particular
question the distribution is invisible.

\begin{theorem}[Leading all the way]\label{thm:race}
For every atomless distribution $\mu$,
\[
\P\left[
\rv{Y}_1+\cdots+\rv{Y}_k>\rv{X}_1+\cdots+\rv{X}_k
\hbox{ for every } 1\le k\le n
\right]
=4^{-n}\binom{2n}{n}.
\]
\end{theorem}

Here is a clean way to obtain this number.  Suppose first that the racers' step distribution is
exponential.  Allow them to continue to time $2n$, and list the first $2n$ of their $4n$ stopping
points in increasing order:
\[
0<\rv{c}_1<\rv{c}_2<\cdots<\rv{c}_{2n}.
\]
The memoryless property says that, as these stopping points are revealed from
left to right, each new point is equally likely to belong to Xavier or to
Yushi, independently of what has already happened.  Thus the ownership record
is just a fair (that is, uniformly random) word of length $2n$ in two letters.
Record an up-step for a Xavier point and a down-step for a Yushi point; then
Yushi leads all the way for the first $n$ turns exactly when this up-down path never dips below $0$.

\begin{lemma}\label{lem:central-binomial-path-count}
Among the $2^{2n}=4^n$ up-down paths of length $2n$, exactly
\(\binom{2n}{n}\) never go below $0$.
\end{lemma}

\begin{proof}
Count paths by their final height.  The number of paths from $0$ to height
$2j$ in $2n$ steps that never go below $0$ is
\[
\binom{2n}{n{+}j}-\binom{2n}{n{+}j{+}1},
\qquad 0\le j\le n,
\]
by the usual reflection argument: the subtracted paths are those whose first
step below $0$ has been reflected.  Summing over all possible final heights
telescopes to
\[
\sum_{j=0}^n
\left(\binom{2n}{n{+}j}-\binom{2n}{n{+}j{+}1}\right)
=\binom{2n}{n}.
\]
\end{proof}

So exponential steps give \(4^{-n}\binom{2n}{n}\). Write
\[
\rv{U}_i=\rv{Y}_i-\rv{X}_i,\qquad
\rv{V}_0=0,\qquad
\rv{V}_k=\rv{U}_1+\cdots+\rv{U}_k.
\]
The differences $\rv{U}_i$ have a symmetric Laplace distribution, and
Yushi leads all the way precisely when
\[
\rv{V}_1>0,\ \rv{V}_2>0,\ \dots,\ \rv{V}_n>0.
\]
Thus the race question is a question about the sequence of positions of a
one-dimensional random walk.

Let $\nu$ be a symmetric atomless probability distribution on $\R$.  Let
$\rv{U}_1,\dots,\rv{U}_n$ be independent samples from $\nu$ and write
\[
\rv{V}_0=0,\qquad \rv{V}_i=\rv{U}_1+\cdots+\rv{U}_i,\qquad 1\le i\le n.
\]
Since $\nu$ is atomless, the values $\rv{V}_0,\dots,\rv{V}_n$ are almost surely
distinct.  They determine a permutation $\rv{\Pi}_n\in S_{0,n}$, where
$S_{0,n}$ denotes the set of permutations of $\{0,\dots,n\}$:
\[
\rv{\Pi}_n(i)=\left|\{0\le r\le n:\rv{V}_r<\rv{V}_i\}\right|.
\]
Throughout the paper, all permutations and patterns are written in zero-based
notation. Thus a permutation with $N$ entries belongs to $S_{0,N-1}$, and a
pattern of length $r$ belongs to $S_{0,r-1}$.
Set
\[
p^n_{i,j}:=\P[\rv{\Pi}_n(i)=j].
\]
For convenience, set $p^n_{i,j}=0$ whenever
$i\notin\{0,\dots,n\}$ or $j\notin\{0,\dots,n\}$.
The random variable $\rv{\Pi}_n(i)$ is the rank of the $i$th partial sum, and
$p^n_{i,j}$ is its \textit{one-point rank}.  In this notation the race event above
is exactly \(\rv{\Pi}_n(0)=0\).  Thus the probability of leading all the way is
the single entry \(p^n_{0,0}\).

The distribution-free behavior of the one-point ranks is classical.  Erd{\H{o}}s and
Kac proved the limiting arcsine law for mean-zero random walks under central-limit
hypotheses \cite{ErdosKac1947PositiveSums}; symmetry is not required for this
limit.  Sparre Andersen proved the exact finite distribution-free law for
symmetric atomless increments \cite{Andersen1949PositiveSums}, and Darling proved a
stronger finite theorem for the ranks of \(0,\rv{V}_1,\dots,\rv{V}_n\)
\cite{Darling1951Sums}.  Related transform formulae for ordered partial sums
were developed by Pollaczek \cite{Pollaczek1952Order} and Wendel
\cite{Wendel1960OrderStatistics}.

\begin{theorem}[Arcsine law \cite{Andersen1949PositiveSums,ErdosKac1947PositiveSums}]\label{thm:arcsine}
For the symmetric atomless walk above, let
\[
\rv{N}_n=\left|\{1\le r\le n:\rv{V}_r>0\}\right|
\]
be the number of positive partial sums.  For $0\le m\le n$,
\[
\P[\rv{N}_n=m]
=
p^n_{0,n-m}
=p^n_{0,m}
=
4^{-n}\binom{2m}{m}\binom{2n-2m}{n-m}.
\]
Consequently, for $0\le x\le 1$,
\[
\lim_{n\to\infty}\P[\rv{N}_n/n\le x]
=
\frac{2}{\pi}\arcsin\sqrt{x}.
\]
\end{theorem}

The symmetry assumption is essential for the exact finite formula: already
for $n=1$ that formula asserts $\P[\rv{U}_1>0]=1/2$, which does not follow from
mean zero.  It is not needed for the limiting assertion.  More generally, if
the increments are iid with
\[
\E[\rv{U}_1]=0,
\qquad
0<\operatorname{Var}(\rv{U}_1)<\infty,
\]
then $\rv{N}_n/n$ converges to the same arcsine distribution, even when the
step distribution is asymmetric \cite{ErdosKac1947PositiveSums}.  Thus
Theorem~\ref{thm:arcsine} has an exact symmetric finite layer and a broader
mean-zero finite-variance limiting layer.

\begin{theorem}[Darling formula and one-point ranks \cite{Darling1951Sums}]\label{thm:main-invariance}\label{thm:darling}\label{thm:explicit-pij}
For every $n$ and every $0\le i,j\le n$, the quantity $p^n_{i,j}$ is
independent of the symmetric atomless distribution $\nu$ and in particular,
\[
p^n_{i,j}
=4^{-n}\sum_{k=0}^j
\binom{2k}{k}\binom{2(i-k)}{i-k}
\binom{2(j-k)}{j-k}\binom{2(n-i-j{+}k)}{n-i-j{+}k},
\]
where a binomial coefficient with a negative lower parameter is interpreted as
$0$.
\end{theorem}

\begin{proof}[Proof of Theorem \ref{thm:race}]
The differences $\rv{U}_i=\rv{Y}_i-\rv{X}_i$ have a symmetric atomless
distribution.  The event that Yushi leads all the way is the event that the
initial value $0$ is below all later partial sums of the difference walk, or
equivalently \(\rv{\Pi}_n(0)=0\).  From Theorems \ref{thm:arcsine} or \ref{thm:main-invariance} we get
\[
p^n_{0,0}=4^{-n}\binom{0}{0}\binom{2n}{n}
=4^{-n}\binom{2n}{n}.
\]
\end{proof}

Within the class of symmetric atomless step distributions, the full
distribution of $\rv{\Pi}_n$ is not independent of $\nu$; only these one-point
ranks are.  Thus such random walks generate non-uniform ordinal patterns, but
Theorem \ref{thm:main-invariance} identifies a distribution-free shadow of
those patterns.  Outside the symmetric class, even the finite one-point ranks
can depend on $\nu$.

The bridge between random walks and the constructions later in the paper is
worth making explicit now.  Given a finite set of positive magnitudes, order
them uniformly and assign independent fair signs; we call the resulting
sequence a random signed permutation of the set.  For a symmetric atomless
random walk, condition on the unordered set of absolute step lengths
\[
A=\{|\rv{U}_1|,\dots,|\rv{U}_n|\}.
\]
Almost surely no two subset sums of $A$ agree; Section~2.2 calls such a set
\emph{generic} and gives the formal definitions.  Conditional on $A$, the
increments form a uniformly random signed permutation of $A$: exchangeability
makes their order uniform, while symmetry makes their signs independent and
fair.  Conversely, Theorem~\ref{thm:generic-set-ranks} will show that every
generic set produces the same one-point ranks $p^n_{i,j}$.  We may therefore
replace random step magnitudes by carefully chosen deterministic ones without
changing the intensity measure.  Powers of two and the binary step sets used
later are two such choices.

\subsection{Patterns, Permutons, and Intensity}

A \textit{pattern} of length $k$ is a permutation in $S_{0,k-1}$ that can be
realized as a normalized sub-permutation of a larger permutation. For distinct
real numbers $x_0,\dots,x_{k-1}$, their zero-based pattern is the permutation
defined by
\[
\pat(x_0,\dots,x_{k-1})(i)
:=\left|\{0\le j\le k-1:x_j<x_i\}\right|.
\]
For example, the pattern $021$ occurs six times in the permutation $04321$,
as the subsequences $043$, $042$, $041$, $032$, $031$, and $021$. If
$\lambda\in S_{0,N-1}$ and $\tau\in S_{0,r-1}$, define the \textit{pattern
density}
\[
d(\lambda,\tau):=\binom{N}{r}^{-1}
\left|\left\{0\le i_0<\cdots<i_{r-1}\le N-1:
\pat(\lambda(i_0),\dots,\lambda(i_{r-1}))=\tau\right\}\right|.
\]
Equivalently, $d(\lambda,\tau)$ is the probability that $r$ uniformly chosen
positions in $\lambda$ induce the pattern $\tau$.  In the example above,
$d(04321,021)=6/\binom{5}{3}=3/5$.

A \textit{permuton} is a Borel probability measure on $[0,1]^2$ with uniform
marginals. Each permuton $\xi$ induces a probability measure
$\P_\xi[\cdot]$ on $S_{0,n}$ by sampling $n{+}1$ independent and identically
distributed (iid) random points from $\xi$, ordering them by their
$x$-coordinates, and recording the relative order of their $y$-coordinates.
For $\tau\in S_{0,n}$, define the permuton pattern density
\[
\tilde d(\xi,\tau):=\P_\xi[\{\tau\}].
\]

To a permutation $\pi\in S_{0,n}$ we associate the \textit{grid permuton}
$\gamma(\pi)$ by dividing $[0,1]^2$ into an $(n{+}1)\times(n{+}1)$ grid and
placing mass $1/(n{+}1)$ uniformly on the cell in column $i$ and row
$\pi(i)$. Note that for a fixed finite permutation $\sigma$, the two finite notions
$d(\sigma,\tau)$ and $\tilde d(\gamma(\sigma),\tau)$ need not be equal.  They
become asymptotically equal as $|\sigma|\to\infty$, but the grid permutation
samples points continuously inside the grid squares, while $d(\sigma,\tau)$
samples vertices.

The \textit{intensity measure} $I\rv{\Xi}$ of a random permuton $\rv{\Xi}$ is
the deterministic permuton defined by
\[
I\rv{\Xi}(E)=\E[\rv{\Xi}(E)].
\]

The following is the Hoppen--Kohayakawa--Moreira--Ráth--Sampaio characterization of convergence for permutations (or permutons).

\begin{theorem}[Permutation limits \cite{HoppenKohayakawaMoreiraRathSampaio2013}]\label{thm:pattern-continuity}
For permutons $\xi_m$ and $\xi$,
\[
\xi_m\to\xi
\]
weakly if and only if
\[
\tilde d(\xi_m,\tau)\to\tilde d(\xi,\tau)
\]
for every finite permutation $\tau$.
\end{theorem}

\begin{lemma}[Random pattern densities {\cite[Theorem 2.5]{BassinoBouvelFerayGerinMaazounPierrot2020Universal}}]\label{lem:random-pattern-equivalence}
For random permutations $\rv{\lambda}_n$ with
$|\rv{\lambda}_n|\to\infty$ in probability and a random permuton
$\rv{\lambda}$, the following are equivalent:
\begin{enumerate}
\item $\gamma(\rv{\lambda}_n)\to\rv{\lambda}$ in distribution;
\item for every finite permutation $\tau$,
\[
d(\rv{\lambda}_n,\tau)\to \tilde d(\rv{\lambda},\tau)
\]
in distribution;
\item for every finite permutation $\tau$,
\[
\E[d(\rv{\lambda}_n,\tau)]\to\E[\tilde d(\rv{\lambda},\tau)].
\]
\end{enumerate}
\end{lemma}

\begin{corollary}\label{cor:random-permuton-determined}
A random permuton is determined in distribution by its expected pattern
densities.
\end{corollary}

{\bf Remark.}
In general,
\[
\E[\tilde d(\rv{\lambda},\tau)]
\ne
\tilde d(\E[\rv{\lambda}],\tau).
\]
For example, let $\rv{\lambda}_n$ be uniform on the cyclic rotations of the
identity permutation in $S_{0,n}$.  Its random limiting permuton is supported
on the graph of a random rotation $x\mapsto x+c\pmod 1$.  Averaging over the
uniform shift $c$ gives the uniform permuton.  Therefore
\[
\tilde d(\E[\rv{\lambda}],\tau)=\frac1{|\tau|!}
\]
for every $\tau$, while
\[
\E[\tilde d(\rv{\lambda},0312)]=0.
\]
Similarly, for the increasing pattern $01$, the finite vertex density and the
grid permuton density differ at finite $n$: the identity has
$d(\mathrm{id},01)=1$, while its grid permuton has
$\tilde d(\gamma(\mathrm{id}),01)$ just below $1$ because two sampled points
can fall in the same vertical strip.

\subsection{An Infinite Family of Permutons with a Common Intensity}

Let $\rv{B}$ be a standard Wiener process and define
\[
F_{\rv{B}}(y):=\Leb\{t\in[0,1]:\rv{B}(t)<y\}.
\]
The random measure $\rv{\mathcal W}$ is given by
\[
\rv{\mathcal W}(E)
:=
\Leb\{t\in[0,1]:(t,F_{\rv{B}}(\rv{B}(t)))\in E\},
\qquad E\subseteq[0,1]^2\ \text{Borel}.
\]
We call $\rv{\mathcal W}$ the \emph{Wiener permuton}.

\begin{theorem}[Random-walk permuton limit]\label{thm:random-walk-permuton-limit}
If the step distribution $\nu$ is atomless with mean $0$ and finite variance,
then
\[
\gamma(\rv{\Pi}_n)\xrightarrow{\mathrm d}\rv{\mathcal W}
\]
as permuton-valued random variables.
\end{theorem}

Let $\rv{M}_{\rm sep}$ denote the \emph{recursive separable permuton} of
\cite{FerayRiveraLopez2023Recursive}.  We interpolate between these two
permutons with a family $\rv{\mathcal M}_L$, indexed by
$L\in\mathbb N_0:=\{0,1,2,\dots\}$, designating
\[
\rv{\mathcal M}_0=\rv{\mathcal W}, ~ \rv{\mathcal M}_\infty=\rv{M}_{\rm sep}.
\]

To state the common-intensity assertion precisely, let $\rv{A}$ and
$\rv{A}'$ be independent arcsine random variables, each with density
\[
a\longmapsto\frac{1}{\uppi\sqrt{a(1-a)}},
\qquad 0<a<1.
\]
For $0\le x\le1$, set
\[
\rv{Y}_x=x\rv{A}+(1-x)\rv{A}'.
\]
Let $\rv{X}$ be uniform on $[0,1]$, independent of $\rv{A},\rv{A}'$, and set
\[
\rv{Y}=\rv{Y}_{\rv{X}}.
\]
Define $\rho$ to be the law of $(\rv{X},\rv{Y})$.  Conditional on
$\rv{X}=x$, the variable $\rv{Y}$ is the sum of two independent scaled
arcsine variables, so $\rho$ has density
\begin{equation}\label{eq:rho-arcsine-convolution}
(x,y)\longmapsto
\frac{1}{\uppi^2}
\int_{\max\{0,x+y-1\}}^{\min\{x,y\}}
\frac{du}{\sqrt{u(x-u)(y-u)(1-x-y+u)}}.
\end{equation}
The density is symmetric in $x$ and $y$.  Since its first marginal is uniform,
its second marginal is uniform as well, and hence $\rho$ is a permuton.  In the
triangular region $0<x<y<1/2$, the bounds in
\eqref{eq:rho-arcsine-convolution} reduce to $0$ and $x$, giving
\[
\frac{1}{\uppi^2}
\int_0^x
\frac{du}{\sqrt{u(x-u)(y-u)(1-x-y+u)}}.
\]
The full density is invariant under
\[
(x,y)\mapsto(y,x),\qquad
(x,y)\mapsto(1-x,1-y),\qquad
(x,y)\mapsto(x,1-y).
\]
At $x=0$, $\rv{Y}_0=\rv{A}'$, so this two-sided rank law reduces to the
limiting arcsine law in Theorem~\ref{thm:arcsine}.  Thus $\rho$ is a
two-dimensional extension of that law.  Theorem~\ref{thm:intensity-limit}
will first identify it as the limit of the universal finite intensities for
symmetric atomless walks; Corollary~\ref{cor:wiener-intensity} will show that
the same limiting intensity holds for atomless asymmetric walks with mean zero
and finite variance.

The family $(\rv{\mathcal M}_L)_{L\in\mathbb N_0\cup\{\infty\}}$ connects
the two endpoints through explicit finite constructions and limiting laws.
For each finite $L$, permutations generated by sparse binary step sets
converge in distribution to $\rv{\mathcal M}_L$.  Every member of the family
has intensity measure $\rho$, the members are mutually singular, and
$\rv{\mathcal M}_L\xrightarrow{\mathrm d}\rv{M}_{\rm sep}$ as
$L\to\infty$.  The precise statement and proof appear in
Theorem~\ref{thm:main-results} in
Section~\ref{subsec:main-permuton-family}.

\begin{figure}[H]
\centering
\includegraphics[height=2.4in]{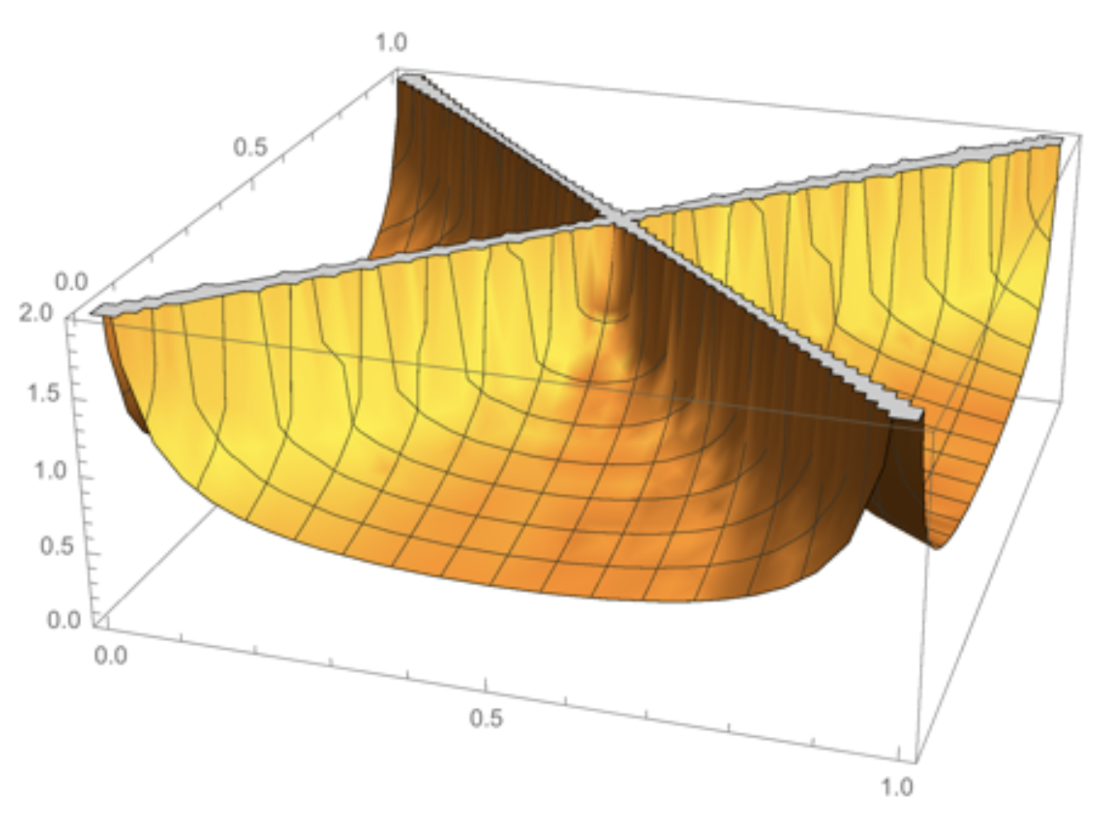}\hfill
\includegraphics[height=2.4in]{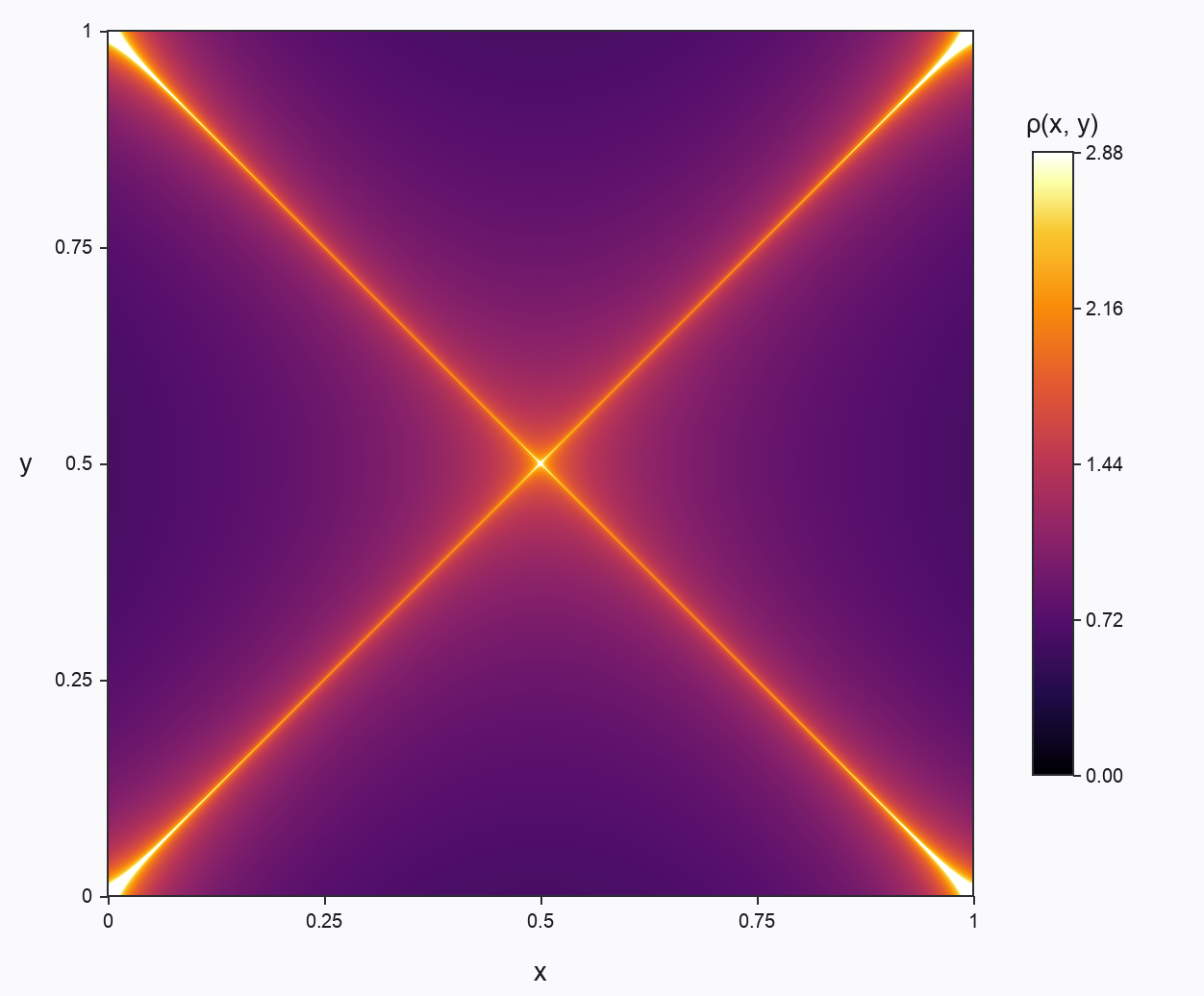}
\caption{The density of the limiting intensity measure $\rho$.
Left: the three-dimensional
plot from \cite{rickpersonal}, with the unbounded densities on the diagonals
truncated. Right: a two-dimensional heat map of the same density.}
\label{rickspic}
\end{figure}

Section~2 first develops the common finite theory for symmetric atomless walks,
connecting random walks, signed permutations, and intensity.  The remaining sections
are organized around the limiting permutons.  Section~3 studies the Wiener
permuton $\rv{\mathcal M}_0=\rv{\mathcal W}$, obtained from atomless mean-zero
finite-variance random walks without a symmetry assumption.  Section~4 studies the recursive separable permuton
$\rv{\mathcal M}_\infty:=\rv{M}_{\rm sep}$, obtained from powers-of-two step
magnitudes and recursive inflation.  Section~5 constructs the intermediate
permutons $\rv{\mathcal M}_L$ and ties the two endpoints together.

\section{Random Walks and Signed Permutations}

We begin with the finite theory shared by all of the limiting models: the
connection between random walks and signed permutations, its combinatorial
consequences, and the universal finite intensity for symmetric atomless
walks.

\subsection{Permutations from Step Distributions}

For a finite sequence $z_0,\dots,z_n$ of distinct real numbers, write
\[
\operatorname{ord}(z_0,\dots,z_n)(i)
=
\left|\{0\le r\le n:z_r<z_i\}\right|.
\]
Then $\rv{\Pi}_n=\operatorname{ord}(\rv{V}_0,\dots,\rv{V}_n)$.  Equivalently, $\rv{\Pi}_n(i)$ is
the rank of the $i$th partial sum, with ranks numbered from $0$ to $n$.

\begin{example}
The event $\rv{\Pi}_n(0)=m$ is the event that exactly $m$ of the nonzero partial
sums $\rv{V}_1,\dots,\rv{V}_n$ are negative.  The distribution of this number is the
discrete arcsine law of Theorem \ref{thm:arcsine}.
\end{example}

\begin{example}\label{ex:full-law-not-invariant}
The full law of $\rv{\Pi}_n$ depends on $\nu$.  Let $\nu_1$ be uniform on
$[-1,1]$ and let $\nu_2$ be uniform on
\[
[-1{-}\eps,-1{+}\eps]\cup[1{-}\eps,1{+}\eps],
\qquad 0<\eps<1/100.
\]
For $\pi=1230$, the law $\nu_1$ gives positive probability, for instance by
taking $\rv{U}_1,\rv{U}_2\in(0,0.2)$ and $\rv{U}_3\in(-1,-0.6)$.  Under $\nu_2$, however,
the first two positive steps cannot be outweighed by one later negative step
of the same approximate size in the required way, so this event has
probability $0$ for sufficiently small $\eps$.
\end{example}

Theorem \ref{thm:main-invariance} follows from Darling's finite rank theorem
\cite{Darling1951Sums}; we include the following self-contained proof.  The
basic decomposition behind the invariance is that
the random walk can be cut at time $i$: points to the left of $\rv{V}_i$ are
described by a reversed walk, and points to the right are described by an
independent forward walk.

\begin{lemma}\label{lem:rank-convolution}
For $0\le i,j\le n$ and symmetric atomless step distribution $\nu$,
\[
p^n_{i,j}=\sum_{k=0}^j p^i_{0,k}\,p^{n-i}_{0,j-k}.
\]
\end{lemma}

\begin{proof}
If $\rv{\Pi}_n(i)=j$, then exactly $j$ of the other partial sums are below $\rv{V}_i$.
Suppose $k$ of these are among $\rv{V}_0,\dots,\rv{V}_{i-1}$ and $j-k$ are among
$\rv{V}_{i{+}1},\dots,\rv{V}_n$.  The increments to the left of $i$, read backward and
with signs changed, form by symmetry a random walk with the same step distribution, and
the increments to the right form an independent such walk.  Thus the
probability of this split is $p^i_{0,k}p^{n-i}_{0,j-k}$, and summing over
$k$ gives the formula.
\end{proof}

\subsection{Permutations from Generic Sets of Reals}

Recall the connection described in the introduction.  For a random walk with a
symmetric atomless step distribution, condition on the unordered set of
absolute step lengths
\[
A=\{|\rv{U}_1|,\dots,|\rv{U}_n|\}.
\]
Almost surely $A$ is generic, and, conditional on $A$, the random-walk
increments form a uniformly random signed permutation of $A$: their order is
uniform by exchangeability, and their signs are independent and fair by
symmetry.  We now reverse this viewpoint by fixing a generic set $A$ and
randomizing only its order and signs.

Let $A=\{a_1,\dots,a_n\}$ be a
set of positive real numbers.  A signed permutation of $A$ is a sequence
\[
S=(s_1,\dots,s_n)=(\eps_1a_{\tau(1)},\dots,\eps_na_{\tau(n)}),
\]
where $\tau$ is a permutation of $\{1,\dots,n\}$ and
$\eps_i\in\{-1,1\}$.  We call $A$ \textit{generic} if no two distinct subsets of $A$
have the same sum.  Equivalently, no nonempty signed sum of elements of $A$ is
zero.

We say that a signed permutation $S=(s_1,\dots,s_n)$ is \textit{$m$-ballot} if exactly $m$ of its
partial sums
\[
s_1,\quad s_1+s_2,\quad \dots,\quad s_1+\cdots+s_n
\]
are negative.

\begin{lemma}\label{lem:generic-ballot}
Let $A$ be generic.  The number of $m$-ballot signed permutations of $A$
depends only on $n=|A|$ and $m$, not on the values of the elements of $A$.
\end{lemma}

\begin{proof}
Write $b_m(A)$ for the number of $m$-ballot signed permutations of $A$.  We
show that $b_m(A)$ is locally constant as the elements of $A$ move through the
space of generic configurations.  Since the generic configurations are
connected after crossing one collision hyperplane at a time, this proves the
claim.

It suffices to change one element, say $a_1$, while keeping the others fixed.
Away from a collision, no partial sum of any signed permutation is zero, so
all signs of all partial sums are unchanged.  At a single collision, exactly
one partial sum changes sign.  Suppose $A'$ is obtained just after the
crossing.  If the signs of the partial sums of a signed permutation
$S'=(s'_1,\dots,s'_n)$ agree with those before the crossing, leave $S'$ fixed.
Otherwise let $k$ be the unique index whose partial sum crossed $0$, and send
\[
(s'_1,\dots,s'_k,s'_{k{+}1},\dots,s'_n)
\longmapsto
(-s'_k,-s'_{k-1},\dots,-s'_1,s'_{k{+}1},\dots,s'_n).
\]
For $j\le k$, the new $j$th partial sum is
\[
-\sum_{r=1}^k s'_r+\sum_{r=1}^{k-j}s'_r,
\]
and for $j>k$ it is
\[
\sum_{r=1}^j s'_r-2\sum_{r=1}^k s'_r.
\]
These signs agree with the signs before the crossing.  The same construction
with the crossing traversed in the opposite direction gives the inverse map.
Thus signed permutations with $m$ negative partial sums are bijected across
the crossing, so $b_m(A)$ is unchanged.
\end{proof}

\begin{corollary}\label{cor:first-rank-invariant}
For every $n$ and $m$, $\P[\rv{\Pi}_n(0)=m]$ is independent of $\nu$ among
symmetric atomless step distributions.
\end{corollary}

\begin{proof}
Condition on the unordered set $A=\{|\rv{U}_1|,\dots,|\rv{U}_n|\}$.  Almost surely this
set is generic.  Given $A$, symmetry makes the signs of the $\rv{U}_i$ independent coin flips, and exchangeability makes their order uniform.  Hence
$(\rv{U}_1,\dots,\rv{U}_n)$ is a uniformly random signed permutation of $A$.  The event
$\rv{\Pi}_n(0)=m$ is exactly the event that exactly $m$ of the partial sums are
negative.  Lemma \ref{lem:generic-ballot} now removes the dependence on $A$.
\end{proof}

\begin{proof}[Proof of Theorem \ref{thm:main-invariance}]
Corollary \ref{cor:first-rank-invariant} proves the invariance of
$p^n_{0,j}$.  Lemma \ref{lem:rank-convolution} writes every $p^n_{i,j}$ as a
finite sum of products of such quantities, and hence every $p^n_{i,j}$ is
independent of $\nu$ within the symmetric atomless class.
\end{proof}

\begin{theorem}[Generic signed permutations]\label{thm:generic-set-ranks}
Let $A$ be any generic set of $n$ positive real numbers, and let $\rv{S}$ be a
uniformly random signed permutation of $A$.  The permutation generated by the
partial sums of $\rv{S}$ has one-point ranks $p^n_{i,j}$.
\end{theorem}

\begin{proof}[Proof of Theorem \ref{thm:generic-set-ranks}]
Write $\rv{S}=(s_1,\dots,s_n)$, set $\rv{V}_0=0$ and
$\rv{V}_k=s_1+\cdots+s_k$ for $1\le k\le n$, and let $\rv{\Pi}_n$ be the
permutation obtained by ranking $\rv{V}_0,\dots,\rv{V}_n$.
Fix $0\leq i\leq n$, and condition on the subset $C\subseteq A$ of magnitudes
that occupy the first $i$ positions of the signed permutation.  Conditional on
$C$, the reversed and sign-changed left block
\[
(-s_i,-s_{i-1},\dots,-s_1)
\]
is a uniformly random signed permutation of $C$, while the right block
\[
(s_{i+1},s_{i+2},\dots,s_n)
\]
is a uniformly random signed permutation of $A\setminus C$.  These two signed
permutations are conditionally independent.

The number of vertices among $\rv{V}_0,\dots,\rv{V}_{i-1}$ that lie below
$\rv{V}_i$ is the
number of negative partial sums of the reversed left block.  Similarly, the
number of vertices among $\rv{V}_{i+1},\dots,\rv{V}_n$ that lie below
$\rv{V}_i$ is the
number of negative partial sums of the right block.  Every subset of a generic
set is generic, so Lemma \ref{lem:generic-ballot} and Corollary
\ref{cor:first-rank-invariant} give, for every choice of $C$,
\[
\P\bigl[\rv{\Pi}_n(i)=j\mid C\bigr]
=\sum_{k=0}^j p^i_{0,k}\,p^{n-i}_{0,j-k}.
\]
The right-hand side does not depend on $C$.  Averaging over $C$ and applying
Lemma \ref{lem:rank-convolution} shows that
\[
\P\bigl[\rv{\Pi}_n(i)=j\bigr]=p^n_{i,j},
\]
as required.
\end{proof}

Thus a generic deterministic step set, sampled without replacement in a
uniformly random order and with independent fair signs, has the same one-point
ranks as a genuine random walk.  The one-point ranks do not detect
which generic set underlies the construction.

\subsection{Universal Finite Intensity for Symmetric Walks}

For the random-walk permutation $\rv{\Pi}_n$ arising from any atomless step
distribution $\nu$, define
\[
\rho^n(\nu):=I\gamma(\rv{\Pi}_n).
\]

\begin{lemma}\label{lem:intensity-matrix}
For every $n$ and every symmetric atomless $\nu$, the measure $\rho^n(\nu)$
is independent of $\nu$ within that class.  On the
$(n{+}1)\times(n{+}1)$ grid, its density in the cell $(i,j)$ is
\[
(n{+}1)p^n_{i,j}.
\]
\end{lemma}

\begin{proof}
The grid permuton $\gamma(\pi)$ has density $n{+}1$ on the cell $(i,\pi(i))$
and density $0$ on the other cells.  Taking expectations gives the
displayed density.  Theorem \ref{thm:main-invariance} removes the dependence
on $\nu$.
\end{proof}

Write $\rho^n$ for this common symmetric-walk intensity.  For an asymmetric
step distribution, $\rho^n(\nu)$ need not equal $\rho^n$.

\begin{theorem}[Random-walk intensity]\label{thm:intensity-limit}
For every symmetric atomless $\nu$,
\[
\rho^n(\nu)=\rho^n\to\rho
\]
in the space of permutons (i.e., weakly as probability measures).
\end{theorem}

\begin{proof}
By Lemma \ref{lem:intensity-matrix}, it is enough to use the explicit formula
for $p^n_{i,j}$.  The formula in Theorem~\ref{thm:darling} gives
$p^n_{i,j}=p^n_{j,i}$.  Reversing the order of the increments gives
$p^n_{i,j}=p^n_{n-i,n-j}$, while negating every increment reverses all
vertical ranks and gives $p^n_{i,j}=p^n_{i,n-j}$.  These identities correspond
respectively to the transformations
\[
(x,y)\mapsto(y,x),\qquad
(x,y)\mapsto(1-x,1-y),\qquad
(x,y)\mapsto(x,1-y),
\]
which generate all eight triangular regions.  Their boundary lines are
immaterial for the density.

Let
\[
a_r:=4^{-r}\binom{2r}{r},\qquad r\geq0.
\]
For a point $(x,y)$ with $0<x<y<1/2$, set
$i=i_n:=\lfloor(n+1)x\rfloor$ and
$j=j_n:=\lfloor(n+1)y\rfloor$.  Thus $i/n\to x$ and $j/n\to y$, and the
formula in Theorem~\ref{thm:darling} becomes
\[
p^n_{i,j}
=\sum_{k=0}^{i}a_k a_{i-k}a_{j-k}a_{n-i-j+k}.
\]
The standard Wallis bound and Stirling estimate give
\[
a_0=1,\qquad
a_r\leq\frac{1}{\sqrt{\uppi r}}\quad(r\geq1),
\qquad
a_r=\frac{1+O(r^{-1})}{\sqrt{\uppi r}}.
\]
In particular, the last estimate is uniform when $r\geq\eta n$ for any
fixed $\eta>0$.

Fix $0<\delta<\delta_0(x,y)$, where $\delta_0(x,y)$ is sufficiently small
for all the bounds below; their constants may be chosen independently of
$\delta$ in this range.  Uniform Stirling estimates on
$\delta n\leq k\leq i-\delta n$ give
\[
(n+1)\sum_{\delta n\leq k\leq i-\delta n}
a_k a_{i-k}a_{j-k}a_{n-i-j+k}
\longrightarrow
\frac{1}{\uppi^2}\int_\delta^{x-\delta}
\frac{dt}
{\sqrt{t(x-t)(y-t)(1-x-y+t)}}.
\]
It remains to control the two singular ends uniformly.  When
$1\leq k\leq\delta n$, the other three indices are bounded below by
positive constant multiples of $n$, so the Wallis bound gives
\[
(n+1)a_k a_{i-k}a_{j-k}a_{n-i-j+k}
\leq \frac{C_{x,y}}{\sqrt{nk}}.
\]
Summing and using $\sum_{k\leq m}k^{-1/2}\leq2\sqrt m$ bounds this part
by $2C_{x,y}\sqrt\delta$.  The term $k=0$ is $O(n^{-1/2})$ after
multiplication by $n+1$.  Applying the same argument with $\ell=i-k$
treats the other endpoint.  Hence the combined contribution of the two
discarded tails has limsup at most $C'_{x,y}\sqrt\delta$.

The integral has the same integrable square-root singularities at $0$ and
$x$, so its discarded tails are also $O_{x,y}(\sqrt\delta)$.  Letting first
$n\to\infty$ and then $\delta\downarrow0$ proves that the grid density
converges at $(x,y)$ to
\[
\frac{1}{\uppi^2}\int_0^x
\frac{dt}
{\sqrt{t(x-t)(y-t)(1-x-y+t)}},
\]
the density of $\rho$ there.

The three symmetries therefore give pointwise convergence of the grid
densities to the density of $\rho$ away from the finitely many boundary
lines of the eight triangular regions, hence almost everywhere on
$[0,1]^2$.  Both are nonnegative probability densities, so Scheff\'e's
lemma yields convergence in $L^1$, and therefore in total variation and
weakly.
\end{proof}

\subsection{Alternative Proofs and Enumerative Formulas}

We retain two alternative derivations of the finite rank formulas, together
with a generating function that makes their symmetries transparent.

\begin{proof}[Alternative recursive proof of Theorem \ref{thm:main-invariance}]
Equation \eqref{eq:rank-convolution-same} below is Lemma
\ref{lem:rank-convolution}:
\begin{equation}\label{eq:rank-convolution-same}
p^n_{i,j}=\sum_{k=0}^j p^i_{0,k}p^{n-i}_{0,j-k}.
\end{equation}
Summing over $i$ gives, for each fixed $j$,
\begin{equation}\label{eq:summed-recurrence}
1=\sum_{i=0}^n p^n_{i,j}
=\sum_{i=0}^n\sum_{k=0}^j p^i_{0,k}p^{n-i}_{0,j-k}.
\end{equation}
Solving \eqref{eq:summed-recurrence} for $p^n_{0,j}$ expresses it in terms of
quantities $p^{n'}_{0,j'}$ with either $n'<n$ or $j'<j$, together with
$p^0_{0,0}=1$.  Induction on $(n,j)$ therefore proves that every
$p^n_{0,j}$ is distribution-free.  The convolution formula
\eqref{eq:rank-convolution-same} then proves the same for every $p^n_{i,j}$.
\end{proof}

As in Lemma \ref{lem:central-binomial-path-count}, we can also frame the
one-point ranks in terms of up-down paths of length $2n$.  Let
$A(n,j)$ be the set of such paths with exactly $2j$ negative edges (edges below the $x$-axis), and let
$B(n,j)$ be the set of such paths whose last visit to $0$ occurs at position
$2j$. We have already seen from Lemma \ref{lem:central-binomial-path-count} that
$|A(n,0)|=\binom{2n}{n}$. The following results are standard; see, for
example, Benjamin and Quinn \cite[Exercise 3f]{benjamin2003proofs}.

\begin{lemma}\label{lem:B-sets-reformatted}
For every $n$, $|B(n,0)|=|B(n,n)|$.
\end{lemma}

\begin{proof}
A path in $B(n,0)$ starts with either an up step followed by an arbitrary
path that never returns to $0$, or a down step followed by the reflection of
such a path.  Hence
\[
|B(n,0)|=2\binom{2n-1}{n-1}=\binom{2n}{n}=|B(n,n)|.
\]
\end{proof}

\begin{proposition}[Negative-edge/last-zero equidistribution]
\label{prop:A-B-bijection-reformatted}
For every $0\le j\le n$,
\[
|A(n,j)|=|B(n,j)|.
\]
\end{proposition}

\begin{proof}
We count $A(n,j)$ by decomposing a path at all of its visits to $0$.  A
\emph{primitive excursion} begins and ends at $0$ and has no intermediate
visit to $0$.  After a sequence of primitive excursions, the path may have a
terminal meander that never returns to $0$.  This decomposition is unique.

Let
\[
C(z):=\sum_{r\geq0}\frac{1}{r+1}\binom{2r}{r}z^r
=\frac{1-\sqrt{1-4z}}{2z}
\]
be the Catalan generating function.  There are
$\frac1r\binom{2r-2}{r-1}$ positive primitive excursions of length $2r$, so
their generating function is $zC(z)$; reflection gives the same count for
negative primitive excursions.  There are
$\binom{2r-1}{r}=\frac12\binom{2r}{r}$ positive terminal meanders of length
$2r$, and equally many negative ones.  Hence the generating function for a
nonempty terminal meander on one specified side is
\[
M(z):=\frac12\left(\frac1{\sqrt{1-4z}}-1\right).
\]

Let $x$ mark half the number of negative edges and let $y$ mark half the
number of positive edges.  The unique decomposition above gives the bivariate
generating function
\begin{align*}
\Psi(x,y)
&:=\sum_{j,k\geq0}|A(j+k,j)|x^jy^k\\
&=\frac{1+M(x)+M(y)}{1-xC(x)-yC(y)}\\
&=\frac1{\sqrt{1-4x}\sqrt{1-4y}}.
\end{align*}
The last equality follows immediately from
$1-2zC(z)=\sqrt{1-4z}$.  Extracting the coefficient of $x^jy^{n-j}$ yields
\[
|A(n,j)|=\binom{2j}{j}\binom{2n-2j}{n-j}.
\]

On the other hand, a path in $B(n,j)$ is uniquely a bridge of length $2j$
followed by a path of length $2n-2j$ that does not return to $0$.  There are
$\binom{2j}{j}$ choices for the bridge and, by Lemma
\ref{lem:B-sets-reformatted}, $\binom{2n-2j}{n-j}$ choices for the terminal
path.  Thus $|B(n,j)|$ is given by the same product.
\end{proof}

\begin{corollary}\label{cor:negative-edge-count}
For every $0\le j\le n$,
\[
|A(n,j)|=\binom{2j}{j}\binom{2n-2j}{n-j}.
\]
\end{corollary}

\begin{proof}
This coefficient was computed in the proof of Proposition
\ref{prop:A-B-bijection-reformatted}.
\end{proof}

\begin{proof}[Proof of the finite rank formula in Theorem \ref{thm:arcsine}] By Corollary \ref{cor:first-rank-invariant}, it suffices to return to the random race where Xavier and Yushi take exponential steps. That means $\rv{U}_i=\rv{Y}_i-\rv{X}_i$ are the steps of a random walk whose step distribution is a centered Laplace distribution. Let
\[
0<\rv{c}_1<\cdots<\rv{c}_{2n}
\]
be the first $2n$ landing places for Xavier and Yushi combined.  The memoryless property implies that the word
recording whether $\rv{c}_r$ comes from Xavier or Yushi is a
uniformly random up-down word (up for Xavier, down for Yushi) of length $2n$.  The event $\rv{\Pi}_n(0)=m$ is
the event that this path has exactly $2m$ negative edges.  Corollary
\ref{cor:negative-edge-count} gives
\[
\left|\{\text{such paths}\}\right|
=
\binom{2m}{m}\binom{2n-2m}{n-m}.
\]
Dividing by $4^n$ gives the formula.  The generic signed-permutation
extension follows from Lemma \ref{lem:generic-ballot}.
\end{proof}

\begin{proof}[Proof of Theorem \ref{thm:arcsine}]
Almost surely none of $\rv{V}_1,\dots,\rv{V}_n$ equals $0$, so
\[
\rv{\Pi}_n(0)
=\left|\{1\le r\le n:\rv{V}_r<0\}\right|
=n-\rv{N}_n.
\]
Consequently, the finite rank formula just proved gives
\[
\P[\rv{N}_n=m]
=p^n_{0,n-m}
=4^{-n}\binom{2m}{m}\binom{2n-2m}{n-m}.
\]
This expression is unchanged when $m$ is replaced by $n-m$, so the same finite
rank formula also gives $p^n_{0,n-m}=p^n_{0,m}$.  The limiting distribution
follows from the finite formula by Stirling's formula and a Riemann-sum
approximation.
\end{proof}

\begin{proof}[Proof of the formula in Theorem \ref{thm:darling}]
Insert the formula for $p^n_{0,k}$ from Theorem \ref{thm:arcsine} into the convolution
identity of Lemma \ref{lem:rank-convolution}.
\end{proof}

\begin{proposition}[\cite{sergipersonal}]\label{prop:generating-function}
The trivariate generating function of the one-point ranks is
\[
\Phi(u,v,z):=\sum_{n,i,j\ge0}p^n_{i,j}u^iv^jz^n
=
\frac{1}{\sqrt{(1-z)(1-uz)(1-vz)(1-uvz)}}.
\]
\end{proposition}

\begin{proof}
Let
\[
\Phi_j(z)=\sum_{n\ge j}p^n_{0,j}z^n,
\qquad
\Phi(v,z)=\sum_{j\ge0}\Phi_j(z)v^j.
\]
Summing Lemma \ref{lem:rank-convolution} over $i$ gives
\[
\frac{z^j}{1-z}
=
\sum_{k=0}^j\Phi_k(z)\Phi_{j-k}(z).
\]
After multiplying by $v^j$ and summing over $j$, this becomes
\[
\Phi(v,z)^2=\frac{1}{(1-z)(1-vz)}.
\]
Since the constant term is positive,
\[
\Phi(v,z)=\frac{1}{\sqrt{(1-z)(1-vz)}}.
\]
Finally, multiplying Lemma \ref{lem:rank-convolution} by $u^iv^jz^n$ and
summing gives
\[
\Phi(u,v,z)=\Phi(v,uz)\Phi(v,z),
\]
which is the stated formula.
\end{proof}

\subsection{Failure of Inversion Invariance}

The symmetry of $u$ and $v$ in Proposition
\ref{prop:generating-function} gives $p^n_{i,j}=p^n_{j,i}$.
One might ask if it is possible to derive this fact by finding a step distribution $\nu$ that allows us to
flip the roles of space and time; more precisely, that induces a distribution on permutations that is
invariant under inversion.  Alas, there is no such $\nu$.

\begin{theorem}[Failure of inversion invariance]\label{thm:inversion}
For every symmetric atomless $\nu$, there exists a permutation
$\pi\in S_{0,3}$ such that
\[
\P[\rv{\Pi}_3=\pi]\ne \P[\rv{\Pi}_3=\pi^{-1}].
\]
\end{theorem}

\begin{proof}
Take $\pi=0312$, so
$\pi^{-1}=0231$.  Let
$\rv{a},\rv{b},\rv{c}$ be the three (signed) step lengths from $\nu$ that generate $\pi$.  Then one must have
$|\rv{a}|>|\rv{b}|>|\rv{c}|$ and $\rv{b}<0<\rv{a},\rv{c}$, giving probability
\[
\frac18\cdot\frac16.
\]
For $0231$, let the three steps be $\rv{a}$, then $\rv{c}$, then $\rv{b}$. We require $\rv{b} < 0 < \rv{a}, \rv{c}$ as before. If $|\rv{a}|>|\rv{b}|>|\rv{c}|$ as before, we get $0231$, as required, but there is an additional contribution when
$|\rv{b}|>|\rv{a}|>|\rv{c}|$ and $|\rv{a}|,|\rv{b}|,|\rv{c}|$ form a triangle.  The latter event has positive
probability for every atomless symmetric distribution: some interval of the form
$[3x,4x]$ has positive mass, and any three lengths in that interval satisfy
the triangle inequalities.  Thus the two probabilities are different.
\end{proof}

Section~4 exhibits a deterministic step set whose induced permutation law is
invariant under inversion.  Its permutations are the random recursive
separable permutations, so the generic-set extension also gives a direct
combinatorial explanation of the symmetry $p^n_{i,j}=p^n_{j,i}$.

\section{The Wiener Permuton: \texorpdfstring{$L=0$}{L=0}}

We now pass from the shared finite theory to the $L=0$ endpoint.  We first
construct the Wiener permuton from a Wiener process and then identify it as
the scaling limit of random-walk permutations with atomless mean-zero
finite-variance increments.

\subsection{Dyadic Construction}

We remind the reader of the standard dyadic construction of the Wiener process (Brownian motion). We start by drawing independent standard normal random variables
\[
\rv{Z}_0,\ \rv{Z}_{0,0},\ \rv{Z}_{1,0},\ \rv{Z}_{1,1},\ \rv{Z}_{2,0},\ \dots\ .
\]
Define $\rv{D}_0$ as the line between $(0,0)$ and $(1,\rv{Z}_0)$. Having constructed $\rv{D}_i$ on the dyadic mesh
\[
\left\{0,\frac1{2^i},\frac2{2^i},\dots,1\right\},
\]
recursively construct $\rv{D}_{i{+}1}$ by keeping the old values. On each interval
\[
[a,b]=[j/2^i,(j{+}1)/2^i],
\]
assign the midpoint $m=(a+b)/2$ the value
\[
\rv{D}_{i{+}1}(m)=\frac{\rv{D}_i(a)+\rv{D}_i(b)}2+\frac{\sqrt{b-a}}2\,\rv{Z}_{i,j}.
\]
Then interpolate linearly between consecutive dyadic points.  The resulting
random functions satisfy the following standard convergence result.

\begin{theorem}[Dyadic convergence to the Wiener process
{\cite[Section 1.1]{MortersPeres2010BrownianMotion}}]
\label{thm:dyadic-wiener-convergence}
The random functions $\rv{D}_i$ almost surely converge uniformly on $[0,1]$
to $\rv{B}$, the standard Wiener process.
\end{theorem}

\begin{lemma}[Atomless occupation measure of the Wiener process
{\cite[Chapter 6]{MortersPeres2010BrownianMotion}}]
\label{lem:wiener-occupation-atomless}
Almost surely, the occupation measure
\[
\omega_{\rv{B}}(E):=\Leb\{t\in[0,1]:\rv{B}(t)\in E\},
\qquad E\subseteq\R\ \text{Borel},
\]
has no atoms.
\end{lemma}

For a real-valued function $g$ on $[0,1]$, we call
\[
F_g(y):=\Leb\{t\in[0,1]:g(t)<y\}.
\]
its \textit{rank transform}.  It records the fraction of time that the graph
of $g$ spends below height $y$.  If $h:[0,1]\to[0,1]$ is measurable, we call
\[
\Gamma(h)(E):=\Leb\{t\in[0,1]:(t,h(t))\in E\},
\qquad E\subseteq[0,1]^2\ \text{Borel}.
\]
the \textit{graph measure} of $h$.  In this notation, the Wiener permuton
defined in the introduction is
\[
\rv{\mathcal W}=\Gamma(F_{\rv{B}}\circ\rv{B}).
\]

Figure~\ref{fig:dyadic-four-rank-transform} shows this operation at a finite
stage of the dyadic construction.  The left panel is the fourth dyadic
approximation \(\rv{D}_4\).  The right panel is the graph of
\(F_{\rv{D}_4}\circ \rv{D}_4\), obtained by replacing each height by the amount
of time that \(\rv{D}_4\) spends below that height.

\begin{figure}[H]
\centering
\includegraphics[width=.98\textwidth]{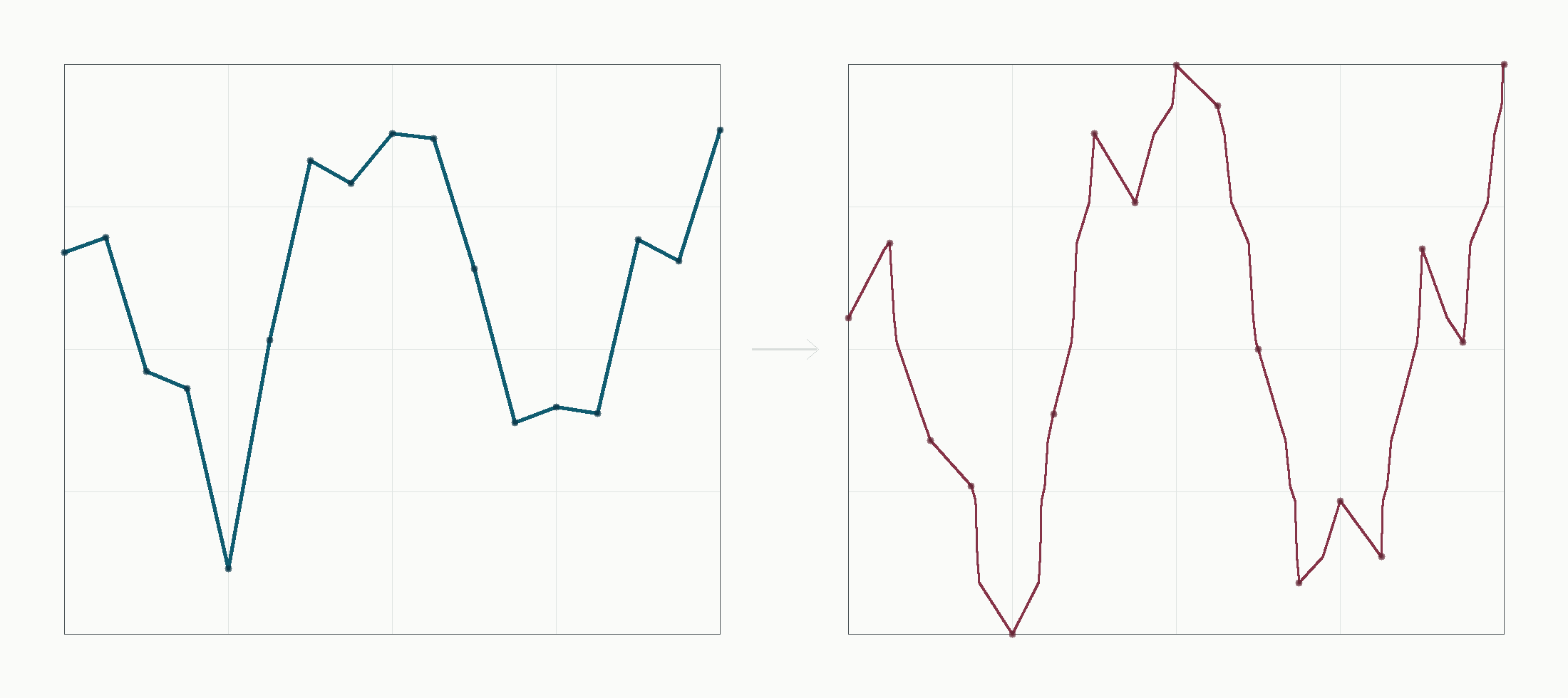}
\caption{The fourth dyadic approximation and its rank transform.  Left:
\(\rv{D}_4\), with the dyadic mesh points marked.  Right:
\(F_{\rv{D}_4}\circ \rv{D}_4\); the graph measure of this curve is the
corresponding finite-stage permuton.}
\label{fig:dyadic-four-rank-transform}
\end{figure}

\begin{figure}[H]
\centering
\includegraphics[width=.6\textwidth]{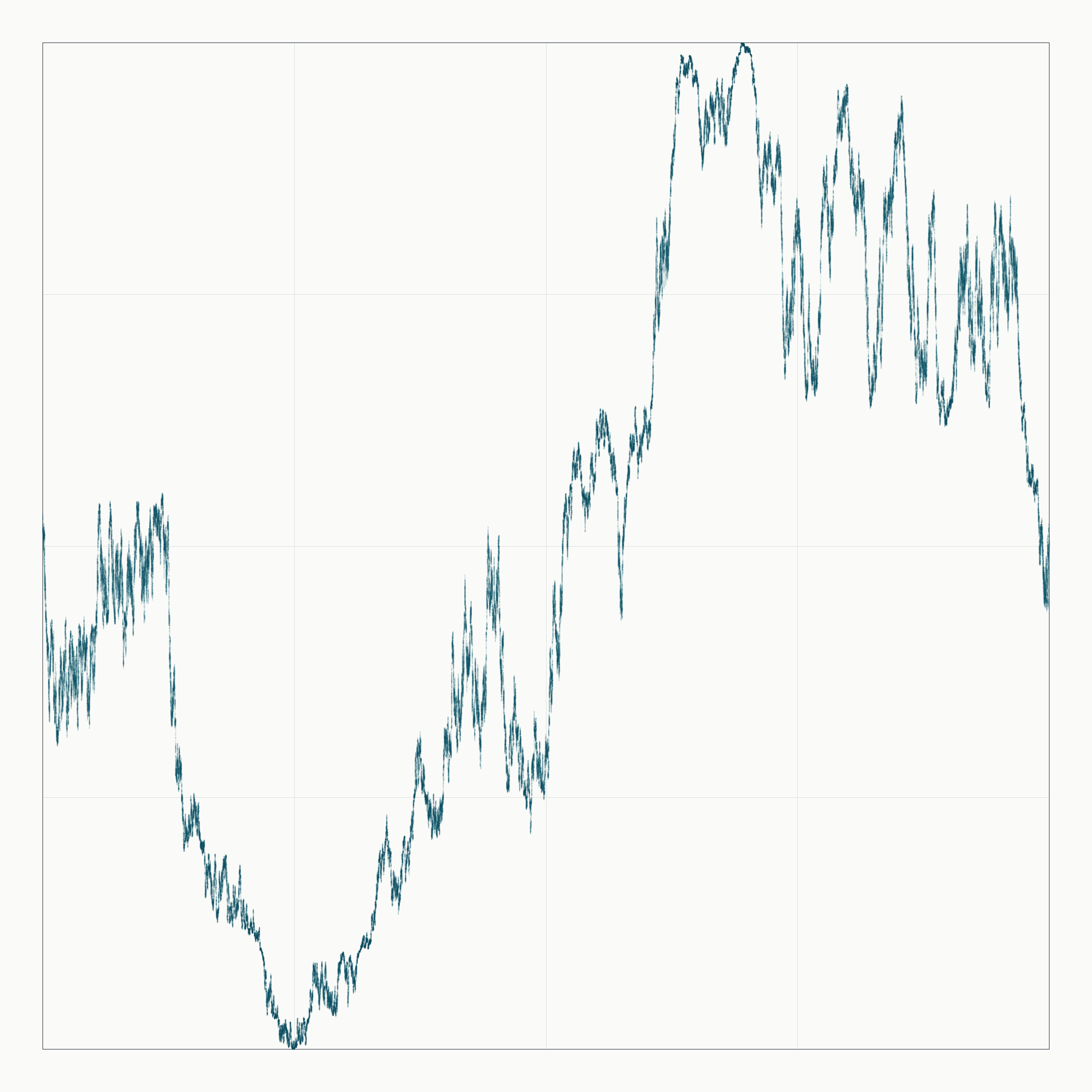}
\caption{A simulation of the Wiener permuton at \(n=250,000\).  The horizontal
coordinate is time; the vertical coordinate is the rank of the Wiener path
height among all sampled times.}
\label{fig:wiener-permuton-simulation}
\end{figure}

\begin{proposition}\label{prop:dyadic-rank-transform-uniform}
Almost surely, the rank-transformed dyadic paths converge uniformly on
$[0,1]$ to the rank-transformed Wiener process:
\[
\left\|F_{\rv{D}_i}\circ \rv{D}_i
-F_{\rv{B}}\circ \rv{B}\right\|_\infty\to0.
\]
\end{proposition}

\begin{proof}
By Theorem~\ref{thm:dyadic-wiener-convergence},
$\rv{D}_i\to\rv{B}$ uniformly almost surely.  By
Lemma~\ref{lem:wiener-occupation-atomless}, the occupation measure of
$\rv{B}$ has no atoms almost surely.  On the event where both statements hold, let
$\delta_i=\|\rv{D}_i-\rv{B}\|_\infty$.  For every $y\in\R$,
\[
F_{\rv{B}}(y-\delta_i)
\le F_{\rv{D}_i}(y)
\le F_{\rv{B}}(y+\delta_i).
\]
The atomlessness of the occupation measure makes $F_{\rv{B}}$ continuous, so
\[
\sup_y\left|F_{\rv{D}_i}(y)-F_{\rv{B}}(y)\right|\to0.
\]
Since $F_{\rv{B}}$ is uniformly continuous on the compact range of the paths,
also
\[
\sup_{t\in[0,1]}
\left|F_{\rv{B}}(\rv{D}_i(t))-F_{\rv{B}}(\rv{B}(t))\right|\to0.
\]
The triangle inequality now gives the result.
\end{proof}

\begin{proposition}\label{prop:dyadic-graph-convergence}
Almost surely,
\[
\Gamma(F_{\rv{D}_i}\circ \rv{D}_i)\to \Gamma(F_{\rv{B}}\circ \rv{B})
\]
in the weak topology on permutons.
\end{proposition}

\begin{proof}
By Proposition~\ref{prop:dyadic-rank-transform-uniform}, almost surely the
functions $F_{\rv{D}_i}\circ\rv{D}_i$ converge uniformly to
$F_{\rv{B}}\circ\rv{B}$.  Fix a realization in this event, and let $\rv{Q}$ be
uniform on $[0,1]$.
For every measurable $h:[0,1]\to[0,1]$, the graph measure $\Gamma(h)$ is the
law of $(\rv{Q},h(\rv{Q}))$.  Thus
$\Gamma(F_{\rv{D}_i}\circ\rv{D}_i)$ and
$\Gamma(F_{\rv{B}}\circ\rv{B})$ are respectively the laws of
\[
\bigl(\rv{Q},(F_{\rv{D}_i}\circ\rv{D}_i)(\rv{Q})\bigr)
\quad\text{and}\quad
\bigl(\rv{Q},(F_{\rv{B}}\circ\rv{B})(\rv{Q})\bigr).
\]
The uniform convergence above implies that the first random point converges
almost surely to the second.  It therefore converges in distribution, which
means that its law converges weakly to the law of the second random point.
These laws are precisely the corresponding permutons.
\end{proof}

The refinement from $\rv{D}_i$ to $\rv{D}_{i+1}$ keeps every old mesh value
and inserts one new midpoint value in each dyadic interval.  Consequently, the
permutation at stage $i{+}1$ refines the permutation at stage $i$: the old
entries retain their relative order, while new entries are inserted between
consecutive old positions.  The following theorem shows that these dyadic
refinements converge to $\rv{\mathcal W}$.

\begin{theorem}[Dyadic permutation process]
\label{thm:dyadic-permutation-process}
Set
\[
\rv{H}_i:=F_{\rv{D}_i}\circ\rv{D}_i,
\qquad t_{i,j}:=\frac{j}{2^i},
\]
and define $\rv{\Theta}_i\in S_{0,2^i}$ by ranking the transformed mesh
values:
\[
\rv{\Theta}_i(j)
=\left|\{0\le k\le 2^i:
\rv{H}_i(t_{i,k})<\rv{H}_i(t_{i,j})\}\right|.
\]
The even-positioned entries of $\rv{\Theta}_{i+1}$ induce $\rv{\Theta}_i$:
\[
\pat\bigl(\rv{\Theta}_{i+1}(0),\rv{\Theta}_{i+1}(2),\dots,
\rv{\Theta}_{i+1}(2^{i+1})\bigr)=\rv{\Theta}_i.
\]
Moreover, almost surely,
\[
\gamma(\rv{\Theta}_i)\to\rv{\mathcal W}.
\]
\end{theorem}

\begin{proof}
Work on the probability-one event from
Proposition~\ref{prop:dyadic-rank-transform-uniform} and
Lemma~\ref{lem:wiener-occupation-atomless}, intersected with the probability-one
event on which, at every finite stage, the mesh heights of $\rv{D}_i$ are
distinct and its linear segments are nonconstant.
Thus $F_{\rv{D}_i}$ is strictly increasing on the range of $\rv{D}_i$, so
ranking the values of $\rv{H}_i$ gives the same permutation as ranking the
corresponding values of $\rv{D}_i$.  Since
$\rv{D}_{i+1}(t_{i+1,2j})=\rv{D}_i(t_{i,j})$, the even-positioned entries of
$\rv{\Theta}_{i+1}$ therefore have relative order $\rv{\Theta}_i$.

Set $\rv{H}:=F_{\rv{B}}\circ\rv{B}$.  By
Proposition~\ref{prop:dyadic-rank-transform-uniform},
\[
\|\rv{H}_i-\rv{H}\|_\infty\to0.
\]
Let $N_i=2^i$ and define the empirical distribution function of the
transformed mesh values by
\[
\widehat G_i(y):=\frac{1}{N_i+1}
\left|\{0\le j\le N_i:\rv{H}_i(t_{i,j})<y\}\right|.
\]
For every continuous $\psi:[0,1]\to\R$, uniform convergence and Riemann sums
give
\[
\frac{1}{N_i+1}\sum_{j=0}^{N_i}\psi(\rv{H}_i(t_{i,j}))
\to\int_0^1\psi(\rv{H}(t))\,dt.
\]
If $\rv{Q}$ is uniform on $[0,1]$, then
$\rv{H}(\rv{Q})=(F_{\rv{B}}\circ\rv{B})(\rv{Q})$ is uniform on $[0,1]$ by
Lemma~\ref{lem:wiener-occupation-atomless} and the probability integral
transform.  Therefore
\[
\sup_{0\le y\le1}|\widehat G_i(y)-y|\to0
\]
by \cite[Lemma 2.11]{vdv1998}.  By the definition of $\rv{\Theta}_i$,
\[
\frac{\rv{\Theta}_i(j)}{N_i+1}
=\widehat G_i(\rv{H}_i(t_{i,j})),
\]
and consequently
\[
\max_{0\le j\le N_i}
\left|\frac{\rv{\Theta}_i(j)}{N_i+1}-\rv{H}_i(t_{i,j})\right|\to0.
\]

Let $\rv{J}_i$ be uniform on $\{0,\dots,N_i\}$, and let
$\rv{S}_i,\rv{T}_i$ be independent uniform random variables on $[0,1]$.
The grid permuton $\gamma(\rv{\Theta}_i)$ is the law of
\[
\left(
\frac{\rv{J}_i+\rv{S}_i}{N_i+1},
\frac{\rv{\Theta}_i(\rv{J}_i)+\rv{T}_i}{N_i+1}
\right).
\]
The preceding uniform estimate and $\|\rv{H}_i-\rv{H}\|_\infty\to0$ show
that this random point converges in distribution to
\[
(\rv{Q},\rv{H}(\rv{Q})),
\]
whose law is
$\Gamma(F_{\rv{B}}\circ\rv{B})=\rv{\mathcal W}$.  Hence
$\gamma(\rv{\Theta}_i)\to\rv{\mathcal W}$ weakly almost surely.
\end{proof}

\subsection{Random-Walk Limit and Its Intensity}

We now prove that the same permuton is the limit for every atomless
mean-zero finite-variance step distribution.

\begin{proof}[Proof of Theorem \ref{thm:random-walk-permuton-limit}]
Let $\rv{U}_1,\rv{U}_2,\dots$ be the iid steps with law $\nu$, set
$\rv{V}_0=0$ and $\rv{V}_i=\rv{U}_1+\cdots+\rv{U}_i$.
For a continuous
path $g$ whose values on the mesh $\{0,1/n,\dots,1\}$ are distinct, let
\[
\sigma_n(g)(i)=\left|\{0\le r\le n:g(r/n)<g(i/n)\}\right|,
\qquad 0\le i\le n,
\]
be the permutation generated by the mesh values of $g$.

Now turn the random walk into a continuous path by placing its rescaled partial
sums at the times $i/n$:
\[
\rv{W}_n(i/n)
=\frac{\rv{V}_i}{\sqrt{n\,\operatorname{Var}(\rv{U}_1)}},
\qquad 0\le i\le n,
\]
and by interpolating linearly between consecutive mesh points.  Donsker's theorem gives
$\rv{W}_n\to \rv{B}$ in distribution as $C[0,1]$-valued random variables
\cite{Billingsley1968Convergence}.  The Skorokhod representation theorem
(see e.g. \cite[p.~70]{Billingsley1968Convergence}) allows us to move to a new probability space and realize copies
$\widetilde{\rv{W}}_n$ and $\widetilde{\rv{B}}$, with the same respective laws
as $\rv{W}_n$ and $\rv{B}$, such that
$\widetilde{\rv{W}}_n\to\widetilde{\rv{B}}$ uniformly almost surely.

On the almost-sure event where this uniform convergence holds, the occupation
measure of $\widetilde{\rv{B}}$ has no atoms, and the mesh values of each
$\widetilde{\rv{W}}_n$ are distinct, define
\[
\widetilde{\rv{H}}_n
:=F_{\widetilde{\rv{W}}_n}\circ\widetilde{\rv{W}}_n,
\qquad
\widetilde{\rv{H}}
:=F_{\widetilde{\rv{B}}}\circ\widetilde{\rv{B}}.
\]
The same rank-transform argument as in
Proposition~\ref{prop:dyadic-rank-transform-uniform} gives
$\widetilde{\rv{H}}_n\to\widetilde{\rv{H}}$ uniformly.  Applying the mesh-rank
argument from the proof of Theorem~\ref{thm:dyadic-permutation-process}, now on
the mesh $\{0,1/n,\dots,1\}$, gives
\[
\gamma(\sigma_n(\widetilde{\rv{W}}_n))
\to
\Gamma(F_{\widetilde{\rv{B}}}\circ\widetilde{\rv{B}})
\]
in distribution.  The distinctness of the mesh values is almost sure because
the step distribution is atomless.  Passing back from the copies to the
original variables does not change the laws.  Finally,
$\gamma(\sigma_n(\rv{W}_n))=\gamma(\rv{\Pi}_n)$: at time
$i/n$ the mesh value is
$\rv{V}_i/\sqrt{n\,\operatorname{Var}(\rv{U}_1)}$, and multiplying all values
by the positive constant $1/\sqrt{n\,\operatorname{Var}(\rv{U}_1)}$ does not
change their relative order.  Since
\[
\Gamma(F_{\widetilde{\rv{B}}}\circ\widetilde{\rv{B}})
\]
has the same distribution as $\rv{\mathcal W}$, this proves
$\gamma(\rv{\Pi}_n)\to\rv{\mathcal W}$ in distribution.
\end{proof}

\begin{corollary}\label{cor:wiener-intensity}
The intensity measure of the Wiener permuton is
\[
I\rv{\mathcal W}=\rho.
\]
More generally, let $\nu$ be any atomless step distribution with mean $0$ and
finite variance, not necessarily symmetric.  If $i_n/(n{+}1)\to x$, then
\begin{equation}\label{eq:generalized-arcsine-rank-limit}
\frac{\rv{\Pi}_n(i_n)}{n{+}1}
\xrightarrow{\mathrm d}
x\rv{A}+(1-x)\rv{A}',
\end{equation}
where $\rv{A},\rv{A}'$ are the independent arcsine variables used to define
$\rho$.  In addition,
\[
\rho^n(\nu)\to\rho
\]
weakly.  Thus the finite measures $\rho^n(\nu)$ may depend on an asymmetric
$\nu$, but their limit does not.
\end{corollary}

\begin{proof}
Put $\rv{H}=F_{\rv{B}}\circ\rv{B}$.  For fixed $x\in[0,1]$, split
$\rv{H}(x)$ into the occupation times before and after $x$:
\[
\rv{H}(x)
=\Leb\{0\le t<x:\rv{B}(t)<\rv{B}(x)\}
+\Leb\{x<t\le1:\rv{B}(t)<\rv{B}(x)\}.
\]
The reversed process
\[
(\rv{B}(x)-\rv{B}(x-t))_{0\le t\le x}
\]
and the forward process
\[
(\rv{B}(x+t)-\rv{B}(x))_{0\le t\le1-x}
\]
are independent Wiener processes on intervals of lengths $x$ and $1-x$.
The Brownian arcsine law \cite[Section~5.4]{MortersPeres2010BrownianMotion}
therefore gives
\[
\rv{H}(x)\stackrel{\mathrm d}{=}
x\rv{A}+(1-x)\rv{A}'.
\]
Consequently, if $\rv{X}$ is uniform on $[0,1]$ and independent of the Wiener
process, the law of $(\rv{X},\rv{H}(\rv{X}))$ is $\rho$ by its definition.
Since
$\rv{\mathcal W}=\Gamma(\rv{H})$, this says exactly that
$I\rv{\mathcal W}=\rho$.

We next prove \eqref{eq:generalized-arcsine-rank-limit}.  The number of partial
sums to the left of time $i_n$ that lie below $\rv{V}_{i_n}$ is
\[
\rv{L}_n
:=\left|\left\{1\le k\le i_n:
\rv{U}_{i_n-k+1}+\cdots+\rv{U}_{i_n}>0
\right\}\right|,
\]
while the number to the right that lie below $\rv{V}_{i_n}$ is
\[
\rv{R}_n
:=\left|\left\{1\le k\le n-i_n:
\rv{U}_{i_n+1}+\cdots+\rv{U}_{i_n+k}<0
\right\}\right|.
\]
Thus $\rv{\Pi}_n(i_n)=\rv{L}_n+\rv{R}_n$.  The two blocks are independent,
and reversing the first block does not change its iid law.  Applying the
mean-zero finite-variance arcsine limit to the first block and to the negated
second block gives
\[
\frac{\rv{L}_n}{n+1}\xrightarrow{\mathrm d}x\rv{A},
\qquad
\frac{\rv{R}_n}{n+1}\xrightarrow{\mathrm d}(1-x)\rv{A}',
\]
jointly, with the evident interpretation when $x=0$ or $x=1$.  This proves
\eqref{eq:generalized-arcsine-rank-limit}.  Notice that no sign symmetry was
used.

Finally, for every bounded continuous function $f$ on $[0,1]^2$, Theorem
\ref{thm:random-walk-permuton-limit} and the definition of intensity give
\[
\int f\,d\rho^n(\nu)
=\E\left[\int f\,d\gamma(\rv{\Pi}_n)\right]
\longrightarrow
\E\left[\int f\,d\rv{\mathcal W}\right]
=\int f\,d\rho.
\]
Hence $\rho^n(\nu)\to\rho$ weakly for every atomless mean-zero
finite-variance $\nu$, including asymmetric laws.
\end{proof}

Equation~\eqref{eq:generalized-arcsine-rank-limit} is the generalized arcsine
law encoded by $\rho$.  At $x=0$ it is the limiting law of the rank of the
initial partial sum, equivalently the limiting arcsine law in
Theorem~\ref{thm:arcsine}; for $0<x<1$ it combines the independent arcsine
occupation times on both sides of $x$.  Exact finite-$n$ universality requires
symmetry; under the standing atomlessness assumption, this generalized
limiting law requires mean zero and finite variance but not symmetry.  For
example, if $\rv{U}=\rv{E}-1$ with $\rv{E}$ exponential of mean
$1$, then $\rv{U}$ is atomless, asymmetric, centered, and has finite variance,
but $\P[\rv{U}>0]=e^{-1}\ne1/2$, so already $\rho^1(\nu)\ne\rho^1$ even
though $\rho^n(\nu)\to\rho$.

\subsection{Wiener Pattern Statistics}

We now turn to the pattern statistics of the limiting random-walk permuton and the Wiener process.

For a fixed Wiener path with atomless occupation measure, applying $F_{\rv{B}}$ to the
Wiener-process values preserves their relative order.  Therefore, for every finite
permutation $\tau$, the $\tau$-density obtained by sampling times from the
Wiener path is the same as $\tilde d(\rv{\mathcal W},\tau)$.  The same
observation applies to any $\rv{D}_i$ after replacing $\rv{B}$ by
$\rv{D}_i$ and $F_{\rv{B}}$ by $F_{\rv{D}_i}$.

\begin{theorem}\label{thm:three-density-convergence}
Let $\rv{H}_i=F_{\rv{D}_i}\circ\rv{D}_i$, and let $\rv{\Theta}_i$ be the
random permutation defined from its transformed mesh values in
Theorem~\ref{thm:dyadic-permutation-process}.
Then, for every finite permutation $\tau$,
\[
\tilde d(\Gamma(\rv{H}_i),\tau),\qquad
d(\rv{\Theta}_i,\tau),\qquad
\tilde d(\gamma(\rv{\Theta}_i),\tau)
\]
all converge almost surely to
\[
\tilde d(\rv{\mathcal W},\tau).
\]
\end{theorem}

\begin{proof}
The first convergence follows from Proposition
\ref{prop:dyadic-graph-convergence} and Theorem
\ref{thm:pattern-continuity}.  The third follows from
Theorem~\ref{thm:dyadic-permutation-process} and again from
Theorem~\ref{thm:pattern-continuity}.

It remains to compare the finite pattern density of $\rv{\Theta}_i$ with the
pattern density of its grid permuton.  If $k=|\tau|$, then sampling $k$ points
from $\gamma(\rv{\Theta}_i)$ agrees with sampling $k$ distinct vertices of
$\rv{\Theta}_i$ except when two sampled points fall in the same vertical strip of
the $(2^i{+}1)\times(2^i{+}1)$ grid.  This exceptional probability is at most
\[
\binom{k}{2}\frac{1}{2^i{+}1}.
\]
On the complementary event, the induced pattern is exactly the pattern
sampled from the vertices of $\rv{\Theta}_i$.  Hence
\[
\left|d(\rv{\Theta}_i,\tau)-\tilde d(\gamma(\rv{\Theta}_i),\tau)\right|
\le
\binom{k}{2}\frac{1}{2^i{+}1}\to0.
\]
\end{proof}

The key simplification for computing expected pattern densities in the Wiener
permuton is that they can be computed from a finite Laplace-increment walk.

\begin{theorem}[Wiener patterns and Laplace walk]\label{thm:laplace-patterns}
Fix $\tau\in S_{0,k}$.  Let $\rv{L}_1,\dots,\rv{L}_k$ be independent Laplace
random variables with density
\[
x\mapsto \frac{1}{\sqrt2}e^{-\sqrt2|x|},
\]
and set
\[
\rv{R}_0=0,\qquad \rv{R}_j=\sum_{\ell=1}^j\rv{L}_\ell.
\]
Then
\[
\E[\tilde d(\rv{\mathcal W},\tau)]
=
\P\left[
\left(j\mapsto\left|\{0\le r\le k:\rv{R}_r<\rv{R}_j\}\right|\right)=\tau
\right].
\]
\end{theorem}

\begin{proof}
Sample $k{+}1$ independent points from $\rv{\mathcal W}$.  Conditional on $\rv{B}$, this
is the same as sampling $\rv{Q}_0,\dots,\rv{Q}_k$ independently and uniformly from
$[0,1]$ and taking
\[
(\rv{Q}_r,(F_{\rv{B}}\circ \rv{B})(\rv{Q}_r))
\]
as our samples. Let $\rv{Q}_{(0)}<\cdots<\rv{Q}_{(k)}$ be these sampled times arranged in
increasing order.  Since $F_{\rv{B}}$ preserves the relative order of
Wiener-process values almost surely,
\[
\E[\tilde d(\rv{\mathcal W},\tau)]
=
\P[(\rv{B}(\rv{Q}_{(0)}),\dots,\rv{B}(\rv{Q}_{(k)}))\text{ has relative order }\tau].
\]
Subtracting $\rv{B}(\rv{Q}_{(0)})$ from all heights does not change the pattern.
Once the sampled times
$\rv{Q}_{(0)}<\cdots<\rv{Q}_{(k)}$ are fixed, the lengths of the consecutive
intervals are
\[
\rv{\Delta}_j=\rv{Q}_{(j)}-\rv{Q}_{(j-1)}
\]
and the corresponding Wiener-process increments
\[
\rv{B}(\rv{Q}_{(j)})-\rv{B}(\rv{Q}_{(j-1)})
\]
are independent normal random variables with variances $\rv{\Delta}_j$.  The
spacings have the following normalized-exponential distribution:
\[
(\rv{\Delta}_1,\dots,\rv{\Delta}_k)
\quad\text{has the same distribution as}\quad
\left(\frac{\rv{E}_1}{\rv{T}},\dots,\frac{\rv{E}_k}{\rv{T}}\right),
\]
where $\rv{E}_0,\dots,\rv{E}_{k{+}1}$ are iid exponential variables of mean $1$ and
\[
\rv{T}:=\sum_{i=0}^{k+1}\rv{E}_i.
\]
Indeed, $\rv{E}_0,\dots,\rv{E}_{k+1}$ can be viewed as the successive waiting
times between points of a rate-one Poisson process.  Given that its
$(k+2)$nd point occurs at time $\rv{T}$, the preceding $k+1$ points are
distributed as ordered uniform points on $[0,\rv{T}]$; dividing their locations
by $\rv{T}$ therefore turns the exponential waiting times into the uniform
spacings above \cite[Chapter~21]{ShorackWellner2009EmpiricalProcesses}.
Multiplying all increments by the common positive factor
$\sqrt{\rv{T}}$ does not change the order of their partial sums. Consequently,
the sampled Wiener values generate the same pattern in distribution as a random
walk with steps
\[
\rv{L}_j=\sqrt{\rv{E}_j}\rv{G}_j,
\]
where $\rv{G}_j$ are iid standard normals independent of the $\rv{E}_j$.  Finally,
\[
\int_0^\infty
\frac{1}{\sqrt{2\uppi v}}e^{-x^2/(2v)}e^{-v}\,dv
=
\frac1{\sqrt2}e^{-\sqrt2|x|},
\]
so the $\rv{L}_j$ are iid Laplace variables with the claimed density.
\end{proof}

The computational point of the theorem is that it gives a simple
answer to the question of how to evaluate expected Wiener pattern densities.
Directly from the definition, one would have to average over Wiener paths and
then sample from the resulting random graph measure.  The theorem replaces
this average by the probability that the entire permutation generated by a
finite Laplace random walk equals the prescribed pattern.  Thus the centered
Laplace distribution is a convenient finite model for computing these
probabilities.

Equivalently, if $P^n_{\nu_{\rm Lap}}$ denotes the law of the permutation
generated by an $n$-step random walk with centered Laplace increments, then
Theorem \ref{thm:laplace-patterns} says that, for every $\pi\in S_{0,n}$,
\[
P^n_{\nu_{\rm Lap}}(\pi)=\E[\tilde d(\rv{\mathcal W},\pi)].
\]
Using the permuton-sampling law $\P_\xi$ defined in the introduction, this is
equivalently
\[
P^n_{\nu_{\rm Lap}}(\pi)
=\E[\P_{\rv{\mathcal W}}(\{\pi\})].
\]
This averaged sampling law is very different from sampling directly
from the deterministic intensity measure
$I\rv{\mathcal W}=\rho$ of Corollary~\ref{cor:wiener-intensity}.
The pattern densities of the mean of a distribution on permutons
are not generally equal to, or even close to, the expected pattern
densities of the permutons.

{\bf Remark.}
\emph{Consecutive patterns}, as opposed to the ``classical'' patterns we have been
talking about, require interval domains.  Since the Laplace walk above itself has
mean zero, finite variance step distribution, we deduce from the above that its
induced permutations have the following property, of interest (perhaps) to the
very active community studying permutation patterns: their limiting consecutive
pattern densities coincide with their expected classical pattern densities.

\section{The Recursive Separable Permuton:
\texorpdfstring{$L=\infty$}{L=infinity}}

We next study the $L=\infty$ endpoint, where powers-of-two magnitudes turn
signed permutations into recursive separable inflations.

The one-point ranks $p^n_{i,j}$ are visible in a construction of separable
permutations.  If $\pi\in S_{0,a}$ and $\sigma\in S_{0,b}$, their sum
$\pi\oplus\sigma\in S_{0,a{+}b{+}1}$ is defined by
\[
(\pi\oplus\sigma)(r)=
\begin{cases}
\pi(r), & 0\le r\le a,\\
(a{+}1)+\sigma(r{-}a{-}1), & a{+}1\le r\le a{+}b{+}1,
\end{cases}
\]
and their skew sum $\pi\ominus\sigma\in S_{0,a{+}b{+}1}$ is defined by
\[
(\pi\ominus\sigma)(r)=
\begin{cases}
(b{+}1)+\pi(r), & 0\le r\le a,\\
\sigma(r{-}a{-}1), & a{+}1\le r\le a{+}b{+}1.
\end{cases}
\]
A permutation is separable if it can be obtained from the one-point
permutation by repeated applications of $\oplus$ and $\ominus$; equivalently,
it avoids the patterns $1302$ and $2031$ \cite{BoseBussLubiw1998PatternMatching}.

Now consider the following random inflation operation: from
$\pi\in S_{0,n}$, choose a value $j\in\{0,\dots,n\}$, increase all values
larger than $j$ by $1$, and replace $j$ by either $(j,j{+}1)$ or $(j{+}1,j)$ with
equal probability.  Starting from the one-point permutation and iterating this
procedure produces the random recursive separable permutations studied by
Féray and Rivera-Lopez \cite{FerayRiveraLopez2023Recursive}.

\begin{proposition}\label{lem:powers-inflation}
Let $A_n=\{1,2,4,\dots,2^{n-1}\}$.  The permutation generated by a uniformly
random signed permutation of $A_n$ has the same law as the permutation obtained
by $n$ recursive random inflations from the one-point permutation.
\end{proposition}

\begin{proof}
We prove the claim by induction on $n$; the case $n=1$ is immediate.  Consider
a signed permutation
\[
S=(s_1,\dots,s_n)
\]
of $A_n$, with partial sums $v_0=0$ and
$v_i=s_1+\cdots+s_i$.  Let $\ell$ be the unique index for which
$s_\ell=\eps\in\{-1,1\}$.  Delete this smallest step and divide every remaining
step by $2$.  The resulting sequence $S'$ is a signed permutation of
$A_{n-1}$, and division by $2$ does not change the relative order of its
partial sums.

The two vertices $v_{\ell-1}$ and $v_\ell=v_{\ell-1}+\eps$ are consecutive in
both time and rank.  Indeed, the largest power in any nonempty signed subsum
exceeds the sum of all smaller powers, so the partial sums are distinct; they
are also integers, so no other partial sum can lie strictly between two values
that differ by $1$.
Contract this pair.  Deleting $s_\ell$ translates every later partial sum by
$-\eps$.  Comparisons within the earlier block or within the later block are
therefore unchanged.  In a comparison between the two blocks, the difference
before and after this translation has the form $E+\eps$ and $E$, respectively,
where $E$ is a nonzero signed sum of distinct elements of
$\{2,4,\dots,2^{n-1}\}$.  Hence $E$ is a nonzero even integer, so $E$ and
$E+\eps$ have the same sign.  Thus no cross-block comparison changes, and the
contracted permutation is exactly the permutation generated by $S'$.

It follows that the permutation generated by $S$ is obtained from the one
generated by $S'$ by inflating the entry at position $\ell-1$.  The inflation
is increasing when $\eps=1$ and decreasing when $\eps=-1$.  For a uniformly
random signed permutation $S$, the position $\ell$ is uniform, the sign
$\eps$ is fair, and both are independent of the uniformly random signed
permutation $S'$.  A uniform position in a permutation corresponds to a
uniform value, so this is exactly one recursive random inflation.  The
induction completes the proof.
\end{proof}

Note that the law of Proposition \ref{lem:powers-inflation} is not the uniform distribution; for example, the probability of the identity permutation and of its reverse is greater than the probability of other separable permutations.

\begin{theorem}[Recursive separable permuton
{\cite[Theorem 1.2 and Corollary 1.10]{FerayRiveraLopez2023Recursive}}]
\label{thm:separable-limit}\label{thm:separable-intensity}
Let $\rv{\Sigma}_n$ be the random recursive separable permutation generated by
successive random inflations.  Then $\gamma(\rv{\Sigma}_n)$ converges almost
surely to a random permuton $\rv{M}_{\rm sep}$, called the
\textit{recursive separable permuton}, and
\[
I\rv{M}_{\rm sep}=\rho.
\]
\end{theorem}

{\bf Remark.}
The recursive separable permuton appears similar to, but differs from, the \textit{Brownian separable
permuton} of
\cite[Section~1.3, Theorem~1.6]{BassinoBouvelFerayGerinMaazounPierrot2020Universal},
which is the limit in distribution of independent uniformly random separable
permutations.  More broadly, Borga's two-parameter family of skew Brownian
permutons connects the Baxter and biased Brownian separable permutons and
describes limits of several models of constrained permutations
\cite{Borga2023SkewBrownianPermuton}.

Proposition \ref{lem:powers-inflation} and Theorem
\ref{thm:generic-set-ranks} provide a bridge between the random-walk
ranks and the recursive separable process: powers of two generate the same
one-point ranks as random walks, and the same permutations as recursive
separable inflation.  Thus the powers-of-two model realizes the
$L=\infty$ endpoint while retaining the common intensity measure $\rho$.

\section{A Permuton Family}

The preceding sections produced two different random permutons with the same
intensity measure $\rho$: the Wiener permuton $\rv{\mathcal W}$, from mean-zero
finite-variance random walks; and the recursive separable permuton
$\rv{M}_{\rm sep}$, from powers-of-two step magnitudes.  We now place them at
the two ends of one construction.

\subsection{Intermediate Permutons from Block Operations}

We begin by defining a weighted, permuton version of the direct and skew sums of permutations
described above.  For $0<r<1$, define the rescaling functions $D_{r,0}$, $D_{r,1}$, $K_{r,0}$, and $K_{r,1}$ by
\begin{align*}
D_{r,0}(x,y)&=(rx,ry),
&D_{r,1}(x,y)&=\bigl(r+(1-r)x,r+(1-r)y\bigr),\\
K_{r,0}(x,y)&=\bigl(rx,1-r+ry\bigr),
&K_{r,1}(x,y)&=\bigl(r+(1-r)x,(1-r)y\bigr).
\end{align*}
The weighted direct and skew sums of two permutons $\alpha$ and $\beta$ are then given by
\begin{align*}
 (\alpha\oplus_r\beta)(E)
 &:=r\,\alpha\bigl(D_{r,0}^{-1}(E)\bigr)
 +(1-r)\,\beta\bigl(D_{r,1}^{-1}(E)\bigr),\\
 (\alpha\ominus_r\beta)(E)
 &:=r\,\alpha\bigl(K_{r,0}^{-1}(E)\bigr)
 +(1-r)\,\beta\bigl(K_{r,1}^{-1}(E)\bigr)
\end{align*}
where $E$ is any Borel subset of $[0,1]^2$.  Both measures are permutons.
\begin{figure}[H]
\centering
\begin{tikzpicture}[x=3cm,y=3cm,font=\small]
\def\cut{0.38}
\def\other{0.62}

\begin{scope}
  \fill[blue!18] (0,0) rectangle (\cut,\cut);
  \fill[orange!22] (\cut,\cut) rectangle (1,1);
  \draw[thick] (0,0) rectangle (1,1);
  \draw[dashed,gray] (\cut,0) -- (\cut,1);
  \draw[dashed,gray] (0,\cut) -- (1,\cut);
  \node at ({\cut/2},{\cut/2}) {$\alpha$};
  \node at ({(1+\cut)/2},{(1+\cut)/2}) {$\beta$};
  \node[above=5pt] at (0.5,1) {$\alpha\oplus_r\beta$};
  \node[below=2pt] at ({\cut/2},0) {$r$};
  \node[below=2pt] at ({(1+\cut)/2},0) {$1-r$};
  \node[left=2pt] at (0,{\cut/2}) {$r$};
  \node[left=2pt] at (0,{(1+\cut)/2}) {$1-r$};
\end{scope}

\begin{scope}[xshift=4.8cm]
  \fill[blue!18] (0,\other) rectangle (\cut,1);
  \fill[orange!22] (\cut,0) rectangle (1,\other);
  \draw[thick] (0,0) rectangle (1,1);
  \draw[dashed,gray] (\cut,0) -- (\cut,1);
  \draw[dashed,gray] (0,\other) -- (1,\other);
  \node at ({\cut/2},{(1+\other)/2}) {$\alpha$};
  \node at ({(1+\cut)/2},{\other/2}) {$\beta$};
  \node[above=5pt] at (0.5,1) {$\alpha\ominus_r\beta$};
  \node[below=2pt] at ({\cut/2},0) {$r$};
  \node[below=2pt] at ({(1+\cut)/2},0) {$1-r$};
  \node[left=2pt] at (0,{\other/2}) {$1-r$};
  \node[left=2pt] at (0,{(1+\other)/2}) {$r$};
\end{scope}
\end{tikzpicture}
\caption{The weighted direct and skew sums.}
\label{fig:weighted-direct-skew-sums}
\end{figure}

Note that if $\pi\in S_{0,a}$ and $\sigma\in S_{0,b}$, and if
\[
r=\frac{a{+}1}{a{+}b{+}2},
\]
then
\[
\gamma(\pi\oplus\sigma)
=\gamma(\pi)\oplus_r\gamma(\sigma),
\qquad
\gamma(\pi\ominus\sigma)
=\gamma(\pi)\ominus_r\gamma(\sigma).
\]
Thus $r$ records the relative block size that is implicit in the finite
notation.

We now define a random sequence of permutons $\rv{\mathcal M}_0, \rv{\mathcal M}_1, \dots$ as follows.
Let $\rv{Q}_0, \rv{Q}_1,\dots$ be independent uniform reals from [0,1], let $\rv{\varepsilon}_0,
\rv{\varepsilon}_1, \dots$
 be independent fair signs, and let $\rv{W}_0, \rv{W}_1,\dots$ be independent random Wiener permutons.
Put $\rv{\mathcal M}_0 = \rv{W}_0$, and suppose $\rv{\mathcal M}_0, \dots,  \rv{\mathcal M}_{L}$
have been defined.  If the copy of $\rv{\mathcal W}$ crossing the vertical line $x=\rv{Q}_L$ occupies
the square $[a,a{+}c] \times [b,b{+}c]$, we replace the contents of that square by the direct sum or
skew sum (according to the sign $\rv{\varepsilon}_L$) of the next two unused Wiener permutons, say $\rv{W}_j$
and $\rv{W}_{j+1}$ (scaled to have total mass $c$). This gives $\rv{\mathcal M}_{L+1}$.

By taking $\rv{R}_L = (\rv{Q}_L - a)/c$ we ensure that $\rv{W}_j \oplus_{\rv{R}_L} \rv{W}_{j+1}$ 
or $\rv{W}_j \ominus_{\rv{R}_L} \rv{W}_{j+1}$ splits correctly at the line $x=\rv{Q}_L$.

The above process is a continuous version of Féray and Rivera-Lopez's inflation process, with shrinking
permutation cells replaced by shrinking Wiener permutons. Observe that with probability 1, the support of a permuton drawn from $\rv{\mathcal M}_L$ has exactly $L{+}1$ connected components. Thus, the random permutons $\rv{\mathcal M}_L$ are not only distinct, but mutually singular.
\begin{theorem}\label{thm:$L$-limit}
Almost surely the random sequence $\rv{\mathcal M}_0, \rv{\mathcal M}_1, \dots \to \rv{M}_{\rm sep}$
in the (weak) permuton topology.
\end{theorem}
\begin{proof}
The largest Wiener copy has side length tending almost surely to zero, since the horizontal intervals are the gaps between successive uniform cut points. Every subsequent refinement preserves the total mass within each existing square. Therefore, integrals against any continuous function form a Cauchy sequence, giving almost-sure weak convergence. The limit depends only on the cuts and signs, not on the Wiener fillings. Identifying this limit with the recursive separable permuton is justified by the same signed-uniform-cut construction in \cite{FerayRiveraLopez2023Recursive}.\end{proof}

An immediate consequence is the weaker statement that the individual random variables
$\rv{\mathcal M}_L$ approach the recursive separable permuton $\rv{M}_{\rm sep}$ in distribution.

Figure~\ref{fig:interpolating-permutons} shows a selection of permutons and their limit from an instance of the
random sequence $\rv{\mathcal M}_0, \rv{\mathcal M}_1, \dots$.

\begin{figure}[H]
\centering
\includegraphics[width=.99\textwidth]{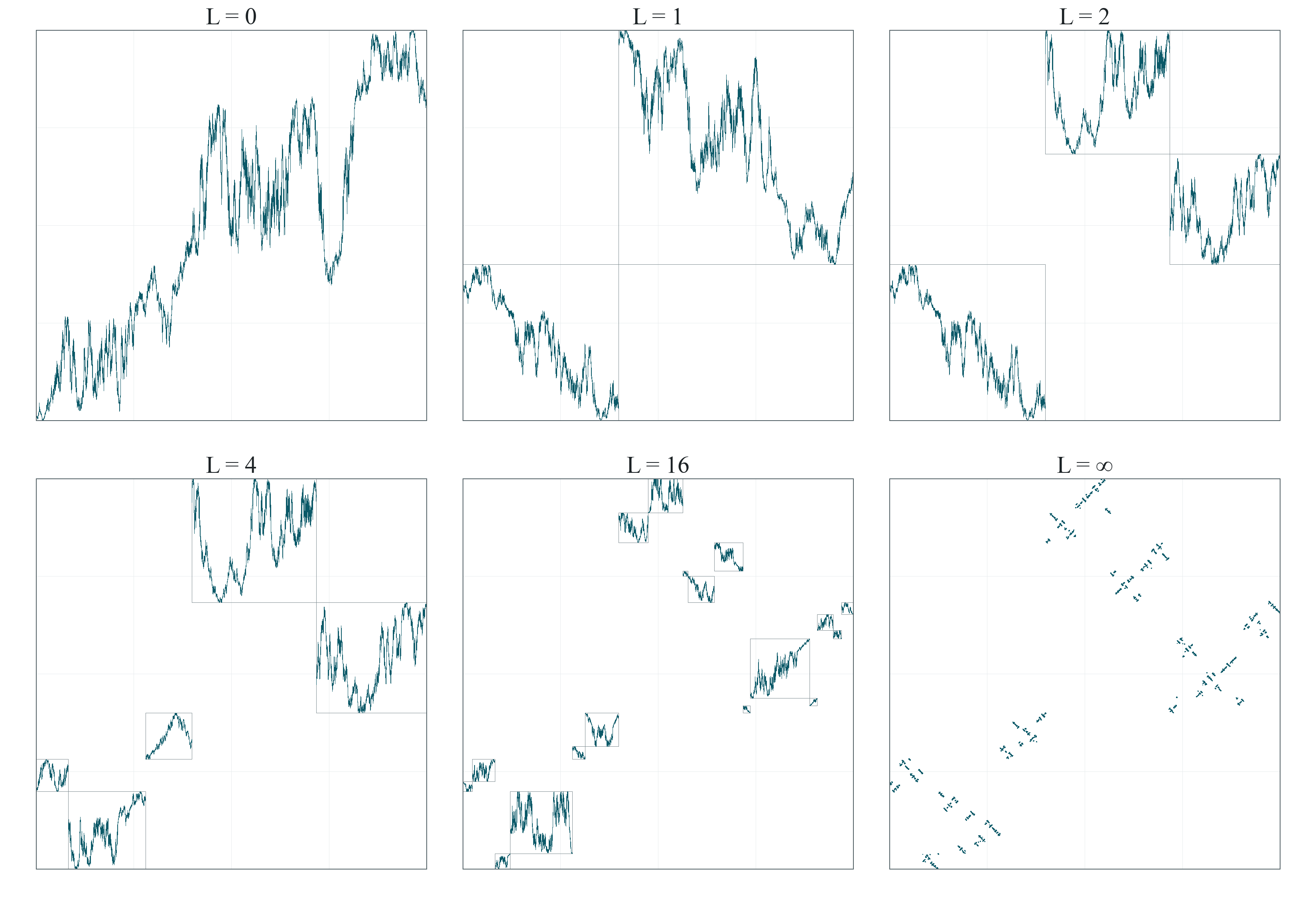}
\caption{Samples from one sequence $\{\rv{\mathcal M}_L\}_{L=0,1,\dots}$ for
$L=0,1,2,4,16$, followed by an approximation of
$\rv{\mathcal M}_\infty=\rv{M}_{\rm sep}$.  Light outlines mark the terminal
squares.}
\label{fig:interpolating-permutons}
\end{figure}

We now want to show that each $\rv{\mathcal M}_L$ is itself the limit of permutations generated by
generic sets of reals, as we already know to be the case for $\rv{\mathcal M}_\infty = \rv{M}_{\rm sep}$.

\subsection{Intermediate Permutons from Binary Step Sets}

Fix $n\geq1$.  For every integer sequence $a=(a_i)_{i\geq0}$ satisfying $|a_i|\leq n$, define
\[
r_n(a):=\sum_{i=0}^{\infty}\frac{a_i}{(n+2)^i}.
\]
The rapidly decreasing weights turn lexicographic comparison of sequences
into ordinary comparison of real numbers. Let $\mathbb N_0=\{0,1,2,\dots\}$. 
The following lemma is easily verified.

\begin{lemma}[Lexicographic order preservation]
\label{lem:binary-lexicographic-order}
Let $A$ be a set of $n$ sequences in $\{-1,0,1\}^{\mathbb N_0}$, and let $a$
and $b$ be two coordinatewise partial sums of a signed ordering of $A$.  Then
\[
r_n(a)<r_n(b)
\quad\Longleftrightarrow\quad
a<_{\rm lex}b.
\]
\end{lemma}


Call a finite set of sequences \emph{sequence-generic} if no nonempty signed
sum of distinct elements is the zero sequence.  For a sequence-generic set
$A\subseteq\{-1,0,1\}^{\mathbb N_0}$ of size $n$, define
\[
\mathcal B_A:=\{|r_n(x)|:x\in A\}.
\]
We call $\mathcal B_A$ a \emph{binary step set}.  Choose a
uniformly random ordering
$\rv{X}^{(1)},\dots,\rv{X}^{(n)}$ of $A$ and independent uniform signs
$\rv{\varepsilon}_1,\dots,\rv{\varepsilon}_n$.  The real-valued steps and
their partial sums are
\[
\rv{U}_j=\rv{\varepsilon}_j r_n(\rv{X}^{(j)}),
\qquad
\rv{V}_0=0,
\qquad
\rv{V}_k=\sum_{j=1}^k\rv{U}_j.
\]
The fixed sign of $r_n(\rv{X}^{(j)})$ can be absorbed into the independent
uniform sign, so these are equivalently signed steps with magnitudes in
$\mathcal B_A$.  The same ordering has sequence-valued partial sums
\[
\rv{S}_0=0,
\qquad
\rv{S}_k=\sum_{j=1}^k
\rv{\varepsilon}_j\rv{X}^{(j)}.  \] Linearity gives $r_n(\rv{S}_k)=\rv{V}_k$. 
Lemma \ref{lem:binary-lexicographic-order} therefore says that the real partial sums ordered in $\mathbb R$
and the sequence-valued partial sums ordered lexicographically induce exactly the same permutation. 
Sequence-genericity also implies that $\mathcal B_A$ is a generic set of positive reals.

{\bf Remark.}
The binary construction is opposite to the powers-of-two construction with
respect to triangle inequalities.  The set
$\mathcal Q_n=\{1,2,4,\dots,2^{n-1}\}$ is superincreasing, so none of its
triples forms a triangle.  By contrast, if
$A\subseteq\{-1,1\}^{\mathbb N_0}$ and $\delta_n:=1/(n+1)$, then
\[
\bigl||r_n(x)|-1\bigr|
\leq |r_n(x)-x_0|
\leq\sum_{i=1}^{\infty}\frac1{(n+2)^i}
=\delta_n.
\]
For $n\geq3$, the inequality $1+\delta_n<2(1-\delta_n)$ therefore implies that
every triple from $\mathcal B_A$ forms a nondegenerate triangle.

\subsection{Representation as a Limit of Permutations}\label{subsec:main-permuton-family}

We first realize $\rv{\mathcal M}_L$ by sparse binary step sets.  Let
$e_i$ be the sequence whose $i$th coordinate is $1$ and whose other
coordinates are $0$.  For $n>L$, define the sequence-generic set
\[
A_{n,L}
:=\{e_0,\dots,e_{L-1}\}
\cup
\{e_L+e_{L+j}:1\leq j\leq n-L\},
\]
where the first set is empty when $L=0$.  Applying $r_n$ gives
\[
\begin{aligned}
\mathcal S_{n,L}
&:=\{r_n(x):x\in A_{n,L}\},\\
\end{aligned}
\]
which is generic by Lemma~\ref{lem:binary-lexicographic-order}. We call $e_0,\dots,e_{L-1}$ the \emph{big steps} and the remaining
$n-L$ elements $e_L+e_{L+j}$ the \emph{small steps}.  Big step
$e_\ell$ is the only step with a nonzero entry in coordinate $\ell$, while
all the small steps first become visible together in coordinate $L$.  Each
small step also has a unique later coordinate, which preserves genericity.

For fixed $L$, let $\rv{\Pi}_{n,L}$ be the permutation generated by a
uniformly random signed permutation of $\mathcal S_{n,L}$.

\begin{theorem}[$\rv{\mathcal M}_L$ as a limit of Permutations]\label{thm:main-results}
For every fixed finite $L\geq0$, the binary step sets
$\mathcal S_{n,L}$ defined above generate random permutations
$\rv{\Pi}_{n,L}$ satisfying
\[
\gamma(\rv{\Pi}_{n,L})
\xrightarrow{\mathrm d}\rv{\mathcal M}_L.
\]
Moreover, every member of the family has intensity measure $\rho$:
\[
I\rv{\mathcal M}_L=\rho,
\qquad L\in\mathbb N_0\cup\{\infty\}.
\]
\end{theorem}

\begin{proof}

For $0\leq\ell<L$, let $\rv{J}_{n,\ell}$ be the position of the big
step $e_\ell$ in the random ordering and let $\rv{\varepsilon}_{n,\ell}$ be its sign.
Since $L$ is fixed,
\[
\left(
\frac{\rv{J}_{n,0}}{n+1},\dots,
\frac{\rv{J}_{n,L-1}}{n+1},
\rv{\varepsilon}_{n,0},\dots,\rv{\varepsilon}_{n,L-1}
\right)
\xrightarrow{\mathrm d}
(\rv{Q}_0,\dots,\rv{Q}_{L-1},
 \rv{\varepsilon}_0,\dots,\rv{\varepsilon}_{L-1}).
\]

We now describe precisely what the big steps do.  Initially regard
the entire consecutive list of partial-sum vertices as one block.  After the
steps $e_0,\dots,e_{\ell-1}$ have been considered, the current blocks are the
consecutive intervals of vertices on which coordinates $0,\dots,\ell-1$ of
the sequence-valued partial sums are constant.  The position
$\rv{J}_{n,\ell}$ lies in one of these blocks.  Within that block, coordinate
$\ell$ is $0$ before the step $e_\ell$ and is
$\rv{\varepsilon}_{n,\ell}$ after it, because no other step has a nonzero
entry in coordinate $\ell$.  Thus $\rv{J}_{n,\ell}$ divides that block into
the consecutive vertices before and after $e_\ell$.

If $\rv{\varepsilon}_{n,\ell}=+1$, every vertex after the division is larger
in lexicographic order than every vertex before it, so the two smaller blocks
form a direct sum.  If $\rv{\varepsilon}_{n,\ell}=-1$, their vertical order is
reversed, so they form a skew sum.  This division of one current block into
these two noninterleaving blocks is what we mean by a \emph{split}.  Repeating
the operation for $\ell=0,\dots,L-1$ gives exactly the first $L$ direct/skew
splits in the definition of $\rv{\mathcal M}_L$.

It remains to identify the limit inside each terminal block.  Index the
terminal blocks from left to right, and let $\rv{K}_{n,b}$ be the number of
small steps in block $b$, for $0\leq b\leq L$.  If
$\rv{P}_0,\dots,\rv{P}_L$ are the spacings of the ordered points
$\rv{Q}_0,\dots,\rv{Q}_{L-1}$, then, jointly with the convergence above,
\[
\left(\frac{\rv{K}_{n,b}+1}{n+1}\right)_{b=0}^{L}
\xrightarrow{\mathrm d}
(\rv{P}_b)_{b=0}^{L}.
\]
Here $\rv{K}_{n,b}+1$ is the number of vertices in block $b$.  Since every
$\rv{P}_b$ is positive almost surely,
$\min_{0\leq b\leq L}\rv{K}_{n,b}\to\infty$ in probability.

Let $\mathcal G_n$ be the sigma-field generated by the unsigned ordering of
all the steps and by the signs of the big steps.  Conditional on
$\mathcal G_n$, the small signs in the terminal blocks are independent fair
signs, while the block lengths $\rv{K}_{n,b}$ are fixed.  For
$\rv{K}_{n,b}\geq1$, let $\rv{W}_{n,b}$ be the linearly interpolated simple
walk in block $b$, with time divided by $\rv{K}_{n,b}$ and height divided by
$\sqrt{\rv{K}_{n,b}}$; define it arbitrarily when $\rv{K}_{n,b}=0$.  The
product form of Donsker's theorem and the preceding divergence of the minimum
block size imply that, for every bounded continuous
$\Phi:(C[0,1])^{L+1}\to\R$,
\[
\E\!\left[
 \Phi(\rv{W}_{n,0},\dots,\rv{W}_{n,L})\,\middle|\,\mathcal G_n
\right]
\xrightarrow{\mathrm P}
\E\!\left[\Phi(\rv{B}_0,\dots,\rv{B}_L)\right],
\]
where $\rv{B}_0,\dots,\rv{B}_L$ are independent standard Wiener processes.
This conditional convergence implies joint convergence with the cut
positions and big-step signs; the limiting Wiener processes are
independent of
$(\rv{Q}_0,\dots,\rv{Q}_{L-1},
\rv{\varepsilon}_0,\dots,\rv{\varepsilon}_{L-1})$.

It remains only to compare these simple walks with the walks using the actual
small-step magnitudes.  After deleting the first $L$ coordinates and multiplying
by $(n+2)^L$, small step $j$ has magnitude $1+(n+2)^{-j}$.  Coupling its sign
with the corresponding simple-walk sign gives, uniformly over every block
and every partial sum,
\[
\left|\text{perturbed partial sum}
      -\text{simple-walk partial sum}\right|
\leq\sum_{j=1}^{n-L}(n+2)^{-j}<\frac1{n+1}.
\]
Consequently, after the normalization in block $b$, the two interpolated
paths differ in sup norm by at most
$1/((n+1)\sqrt{\rv{K}_{n,b}})$ whenever $\rv{K}_{n,b}\geq1$.
The Skorokhod and mesh-rank argument from the proof of
Theorem~\ref{thm:random-walk-permuton-limit}, applied simultaneously to the
$L+1$ blocks, therefore shows that the grid permutons induced by the actual
small-step walks converge jointly to independent copies
$\rv{\mathcal W}_0,\dots,\rv{\mathcal W}_L$ of the Wiener permuton,
independently of the limiting cuts and signs.

Finally, lexicographic splitting gives an exact iterated direct/skew
decomposition of $\gamma(\rv{\Pi}_{n,L})$ whose leaf masses are
$(\rv{K}_{n,b}+1)/(n+1)$.  The maps
\[
(r,\alpha,\beta)\longmapsto\alpha\oplus_r\beta,
\qquad
(r,\alpha,\beta)\longmapsto\alpha\ominus_r\beta
\]
are continuous for $0<r<1$ in the weak topology: this follows directly from
their defining integrals against continuous test functions.  Hence every
finite iterated assembly is continuous at cut vectors with distinct
coordinates in $(0,1)$, an event of probability one.  The continuous mapping theorem
now identifies the assembled limit with the construction of
$\rv{\mathcal M}_L$.

We now check that all the $\rv{\mathcal M}_L$ have intensity measure $\rho$. 
For finite $L$, the set $\mathcal S_{n,L}$ is generic.  Theorem
\ref{thm:generic-set-ranks} and Lemma~\ref{lem:intensity-matrix} therefore
give
\[
I\gamma(\rv{\Pi}_{n,L})=\rho^n,
\]
where $\rho^n$ is the common symmetric-walk finite intensity.  Theorem
\ref{thm:main-results} and convergence in distribution gives convergence of expectations against bounded continuous test functions, which shows that the left side converges weakly to
$I\rv{\mathcal M}_L$, while Theorem~\ref{thm:intensity-limit} gives
$\rho^n\to\rho$.  Hence $I\rv{\mathcal M}_L=\rho$.  The case $L=\infty$ is
Theorem~\ref{thm:separable-intensity}. \end{proof}

\section{Conclusion}

We have studied three ways of producing random permutations: random walks with
symmetric atomless increments, uniformly random signed permutations of generic
sets, and random walks with atomless mean-zero finite-variance increments.  For
every $n$, the first two have the same one-point ranks $p^n_{i,j}$ and
hence the same finite intensity $\rho^n$; these common intensities converge to
$\rho$.  An asymmetric walk in the third class need not have these finite
one-point ranks.  Nevertheless, its rescaled one-point ranks obey the
same generalized arcsine limit
\eqref{eq:generalized-arcsine-rank-limit}, and its intensity measures satisfy
$\rho^n(\nu)\to I\rv{\mathcal W}=\rho$.  Thus $\rho$ is a two-dimensional
extension of the classical arcsine law.  Just as the mean-zero finite-variance
arcsine limit requires no symmetry, neither does this generalized law.

At the limiting level, the Wiener permuton, the recursive separable permuton,
and every intermediate permuton $\rv{\mathcal M}_L$ have the same intensity
measure $\rho$. The Laplace representation
provides a finite method for computing the expected pattern densities of the
Wiener endpoint, while the binary hierarchy displays explicitly how many
different random laws can lie above one intensity measure.

\section*{Acknowledgments}

We thank Evita Nestoridi for getting us started on random races, Jim Pitman for historical references, and both for very useful conversations.

\newpage 
\begingroup
\small
\bibliographystyle{alpha}
\bibliography{bib}
\endgroup


\end{document}